%% file: ClusterP_11_Mar.tex
\documentclass[a4paper,12pt,oneside,reqno]{amsart}
\usepackage{amssymb,amsfonts,amsmath,amsthm,amscd, bbm}
\usepackage{graphicx, color}
\numberwithin{equation}{section}  

\newcommand{\beq}{\begin{equation}} 
\newcommand{\eeq}{\end{equation}} 
\newcommand{\bea}{\begin{aligned}}
\newcommand{\eea}{\end{aligned}}
\newcommand{\bdm}{\begin{displaymath}}
\newcommand{\edm}{\end{displaymath}}
\newcommand{\barr}{\begin{array}}
\newcommand{\earr}{\end{array}}
\newcommand{\ben}{\begin{enumerate}}
\newcommand{\een}{\end{enumerate}}
\newcommand{\bde}{\begin{description}}
\newcommand{\ede}{\end{description}}

\newtheorem{teor}{Theorem}[section]
\newtheorem{prop}[teor]{Proposition}
\newtheorem{lem}[teor]{Lemma}
\newtheorem{cor}[teor]{Corollary}

\newtheorem{rem}[teor]{Remark}

\newcommand{\R}{\mathbb{R}}

\newcommand{\1}{{\mathbbm 1}}

\newcommand{\N}{\mathbb{N}}

\newcommand{\PP}{\mathbb{P}}
\newcommand{\E}{{\mathbb{E}}}

\newcommand{\defi}{\equiv} 
\newcommand{\law}{\stackrel{\text{law}}{=}}

\newcommand{\de}{\delta}
\newcommand{\dd}{\text{d}}
\newcommand{\ee}{\text{e}}

\newcommand{\w}{\omega}

\newcommand{\vare}{\varepsilon}

\begin{document}
\title[The extremal process of BBM]{The Extremal Process of Branching Brownian Motion. }
\author[L.-P. Arguin]{Louis-Pierre  Arguin}            
 \address{L.-P. Arguin\\ Courant Institute of Mathematical Sciences \\
New York University \\
251 Mercer St. New York, NY 10012}
\email{arguin@math.nyu.edu}
\author[A. Bovier]{Anton Bovier}
\address{A. Bovier\\Institut f\"ur Angewandte Mathematik\\Rheinische
   Friedrich-Wilhelms-Uni\-ver\-si\-t\"at Bonn\\Endenicher Allee 60\\ 53115
   Bonn,Germany}
\email{bovier@uni-bonn.de}

\author[N. Kistler]{Nicola Kistler}
\address{N. Kistler\\Institut f\"ur Angewandte Mathematik\\Rheinische
   Friedrich-Wilhelms-Uni\-ver\-si\-t\"at Bonn\\Endenicher Allee 60\\ 53115
   Bonn,
Germany}
\email{nkistler@uni-bonn.de}


\thanks{L.-P. Arguin is supported by the NSF grant DMS-0604869. A. Bovier is 
partially supported through the German Research Council in the SFB 611 and
the Hausdorff Center for Mathematics. N. Kistler is partially supported by the Hausdorff Center for Mathematics. Part of this work has been carried out during the {\it Junior Trimester Program Stochastics} at the Hausdorff Center in Bonn, and during a visit by the third named author to the Courant Institute at NYU: hospitality and financial support of both institutions are gratefully acknowledged.}

 \date{\today}

\begin{abstract} 
We prove that the extremal process of branching Brownian motion, in the limit of large times, converges weakly to a cluster point process. 
The limiting process is a (randomly shifted) Poisson cluster process, where the positions of the clusters is a Poisson process with exponential density.
The law of the individual clusters is characterized as branching Brownian motions conditioned to perform "unusually large displacements", and its existence is proved.
The proof combines three main ingredients.
 First, the results of Bramson on the convergence of solutions of the Kolmogorov-Petrovsky-Piscounov equation with general initial conditions to standing waves. 
 Second, the integral representations of such waves as first obtained by Lalley and Sellke in the case of Heaviside initial conditions. 
 Third, a proper identification of the tail of the extremal process with an auxiliary process, which fully captures the large time asymptotics of the extremal process. 
 The analysis through the auxiliary process is a rigorous formulation of the {\it cavity method} developed in the study of mean field spin glasses. 
\end{abstract}

\maketitle
\tableofcontents

\section{Introduction} 
Branching Brownian Motion (BBM) is a continuous-time Markov branching process which plays an important role in the theory of partial differential equations \cite{aronson_weinberger,aronson_weinberger_two, kpp, mckean}, in the theory of disordered systems \cite{ BovierKurkova_II, derrida_spohn}, and in biology \cite{fisher}. It is constructed as follows. 

Start with a single particle which performs standard Brownian Motion $x(t)$ 
with $x(0)=0$, which it continues for an exponential holding time $T$ independent of $x$, with $\PP\left[ T > t\right] = \ee^{-t}$. At time $T$, the particle splits independently of $x$ and $T$ into $k$ offsprings with probability $p_k$, where $\sum_{k=1}^\infty p_k = 1$, $\sum_{k=1}^\infty k p_k = 2$, and $K \defi \sum_{k} k(k-1) p_k < \infty$. These particles continue along independent Brownian paths starting at $x(T)$, and are subject to the same splitting rule, with the effect that the resulting tree $\mathfrak X$ contains, after an elapsed time $t>0$, $n(t)$ particles located at $x_1(t), \dots, x_{n(t)}(t)$, with $n(t)$ being the random number of particles generated up to that time (it holds that $\E n(t)=e^t$). 

The link between BBM and partial differential equations is provided by the following observation due to McKean \cite{mckean}: if one denotes by
\beq \label{bbm_repr}
u(t, x) \defi \PP\left[ \max_{1\leq k \leq n(t)} x_k(t) \leq x \right]
\eeq
the law of the maximal displacement, a renewal argument shows that $u(t,x)$ solves
the Kolmogorov-Petrovsky-Piscounov or Fisher [F-KPP] equation, 
\beq \bea \label{kpp_equation}
& u_t = \frac{1}{2} u_{xx} + \sum_{k=1}^\infty p_k u^k -u, \\
& u(0, x)= 
\begin{cases}
1, \; \text{if}\;  x\geq 0,\\
0, \, \text{if}\; x < 0. 
\end{cases}
\eea \eeq 
The F-KPP equation admits traveling waves: there exists a unique solution satisfying 
\beq \label{travelling_one}
u\big(t, m(t)+ x \big) \to \omega(x) \qquad \text{uniformly in}\;  x\; \text{as} \; t\to \infty,
\eeq
with the centering term given by
\beq \label{centering_kpp}
m(t) = \sqrt{2} t - \frac{3}{2 \sqrt{2}} \log t, 
\eeq
and $\w(x)$ the distribution function which solves the o.d.e. 
\beq \label{wave_pde}
\frac{1}{2} \omega_{xx} + \sqrt{2} \omega_x + \omega^2 - \omega = 0.
\eeq
If one excludes the trivial cases, solutions to \eqref{wave_pde} are unique up to translations: this will play a crucial role in our considerations. \\

Lalley and Sellke \cite{lalley_sellke} provided a characterization of the limiting law of the maximal displacement in terms of a {\it random shift} of the Gumbel distribution. 
More precisely, denoting by 
\begin{equation}
\label{eqn: deriv}
Z(t) \defi \sum_{k=1}^{n(t)} \left( \sqrt{2}t -x_k(t) \right) \exp-\sqrt{2}\left( \sqrt{2}t -x_k(t) \right),
\end{equation}
the so-called {\it derivative martingale}, Lalley and Sellke proved that $Z(t)$ converges almost surely to a strictly positive 
random variable $Z$, and established the integral representation 
\beq \label{gumbel_like}
\omega(x) = \E\left[ \exp\left(- C Z \ee^{-\sqrt{2}x} \right)\right],
\eeq
for some $C>0$. 

It is  also  known (see e.g. Bramson \cite{bramson} and Harris \cite{harris})
 that
\beq \label{to_the_right}
\lim_{x\to\infty}\frac{1-\omega(x)}{x \ee^{-\sqrt{2} x}}=C. 
\eeq
 (The reason why the
 two constants $C$ in \eqref{gumbel_like} and \eqref{to_the_right} are equal is
 apparent in \cite{lalley_sellke}, and \cite{abk_poissonian}). \\

Despite the precise information on the maximal  displacement of BBM, an understanding of the full statistics of the largest particles
is still lacking. The statistics of such particles are fully encoded in the extremal process, the random measure:
\beq \label{extremal_process_def}
{\mathcal E}_t \defi \sum_{k\leq n(t)} \delta_{x_k(t) -m(t)}.
\eeq
Few papers have addressed so far the large time limit of the extremal process of branching Brownian motion.

On the physical literature side, we mention the contributions by Brunet and Derrida \cite{brunet_derrida_first, brunet_derrida_second}, 
who reduce the problem of the statistical properties of particles "at the edge" of BBM to that of identifying the finer properties of the delay of travelling waves. 
(Here and below, "edge" stands for the set of particles which are at distances of order one from the maximum).

On the mathematical side, properties of the large time limit of the extremal process have been established in two papers of ours \cite{abk, abk_poissonian}. In a first paper we obtained a precise description of the {\it paths} of extremal particles which in turn imply a somewhat surprising restriction of the correlations of particles at the edge of BBM.
These results were instrumental in our second paper on the subject where we proved that a certain process obtained by a correlation-dependent thinning of the extremal particles converges to a random shift of a Poisson Point Process (PPP) with exponential 
density. 

It is the purpose of this paper to complete this picture and to provide an explicit construction of the extremal process of branching Brownian motion in 
the limit of large times. 
We prove that the limit is a randomly shifted Poisson cluster process.
Up to a realization-dependent shift, this point process
corresponds to the superposition of independent point processes or {\it clusters}. The maxima of these point processes, or {\it cluster-extrema}, form a Poisson point process with exponential density.
Relative to their maximum, the laws of the individual clusters are identical.
The law of the clusters coincides with that of a branching Brownian motion conditioned to perform "unusually high jumps". The precise statement is given in Section \ref{main_results}.\\

Understanding the extremal process of BBM is a longstanding problem of
fundamental interest. Most results concerning extremal processes of 
correlated random variables concern criteria that show that it behaves as if
there were no correlations \cite{leadbetter}. Bramson's result shows that this cannot be the 
case for BBM. A class of models where a more complex structure of 
\emph{Poisson 
cascades} was shown to emerge are the \emph{generalized random energy models}
of Derrida \cite{derrida_grem,BovierKurkova}. These models, however, have a 
rather simple hierarchical structure involving a finite number of levels only
which greatly simplifies the analysis, which cannot be carried over to models
with infinite levels of branching such as BBM or the \emph{continuous random 
energy models} studied in \cite{BovierKurkova_II}.  BBM is a case right at the 
borderline where correlations just start to effect the extremes and the
structure of the 
extremal process. Our results thus allow to peek into the 
world beyond the simple Poisson structures and hopefully open the gate 
towards the rigorous understanding of complex extremal structures. 
Mathematically, BBM offers a spectacular interplay between probability
and non-linear pdes, as was noted already by McKean. We 
will heavily rely on this dual way to attack and understand this problem. \\

The remainder of this paper is organised as follows. In Section 2 we state
our main results, the heuristics behind them, and we indicate the major 
steps in the proof. In Section 3 we give the details of the proofs. 

\section{Main result} \label{main_results}
We first recall from \cite{kallenberg} the standard infrastructure for the study of point process. Let $\mathcal{M}$ be the space of Radon measure on $\R$. Elements of $\mathcal{M}$ are in correspondence with the positive linear functionals
on $\mathcal{C}_c(\R)$, the space of continuous functions on $\R$ with compact support. In particular, any element of $\mathcal{M}$ is locally finite.
The space $\mathcal{M}$ is endowed with the vague topology (or 
 weak-$\star$-topology), that is, $\mu_n\to \mu$ in $\mathcal{M}$ if and only if 
for any $\phi\in \mathcal{C}_c(\R)$, $\int \phi d\mu_n\to \int \phi d\mu$. The law of a random element $\Xi$ of $\mathcal{M}$, or random measure, is determined
by the collection of real random variables $\int \phi d\Xi$, $\phi \in \mathcal{C}_c(\R)$. A sequence $(\Xi_n)$ of random elements of $\mathcal{M}$ is said
to converge to $\Xi$ if and only if for each $\phi \in \mathcal{C}_c(\R)$,  the random variables $\int \phi d\Xi_n$ converges in the weak sense to $\int \phi d\Xi$.
A point process is a random measure that is integer-valued almost surely. It is a standard fact that point processes are closed in the set of random elements of
$\mathcal{M}$. 

The limiting point process of BBM is constructed as follows.
Let $Z$ be the limiting derivative martingale.
Conditionally on $Z$, we consider the Poisson point process (PPP) of density $C Z \sqrt{2}\ee^{-\sqrt{2}x} \dd x$:
\beq \label{ppp_expo}
P_Z\equiv \sum_{i\in \N}\delta_{p_i} \defi \text{PPP}\left(C Z \sqrt{2}\ee^{-\sqrt{2}x} \dd x \right), 
\eeq
with $C$ as in \eqref{gumbel_like}.
Now let $\{x_k(t)\}_{k\leq n(t)}$ be a BBM of length $t$. 
Consider the point process of the gaps $\sum_k\delta_{x_k(t)-\max_j x_j(t)}$ conditioned on the event $\{\max_j x_j(t)- \sqrt{2}t >0\}$. 
Remark that, in view of \eqref{centering_kpp}, the probability that the maximum of BBM shifted by $-\sqrt{2}t$ does not drift to $-\infty$ is vanishing in the large time limit.
In this sense, the BBM is conditioned to perform "unusually large displacements". 
It will follow from the proof that the law of this process exists in the limit. 
Write $\mathcal D =\sum_j \delta_{\Delta_j}$ for a point process with this law and consider independent, identically distributed (iid) copies $(\mathcal D^{(i)})_{i\in\N}$.
The main result states that the extremal process of BBM as a point process converges as follows:
\begin{teor}[Main Theorem] Let $P_Z$ and $\mathcal D^{(i)} = \{\Delta_j^{(i)}\}_{j\in \N}$ be defined as above. 
Then the family of point processes $\mathcal {E}_t$, defined in \eqref{extremal_process_def}, converges in distribution to a point process, $\mathcal E$, 
given by
\beq\label{mt.1}
\mathcal E\defi\lim_{t\to \infty} {\mathcal E}_t \law \sum_{i,j}\delta_{p_i + \Delta_j^{(i)}}.
\eeq
\label{teor main}
\end{teor}
We remark in passing that a similar structure is expected to emerge in all the models which are conjectured to fall into the universality class
of branching Brownian motion, such as the {\it 2-dim Gaussian Free Field} \cite{bolthausen_deuschel_giacomin, bolthausen_deuschel_zeitouni, bramson_zeitouni}, or the {\it cover time} for the simple random walk on the two dimensional discrete torus \cite{dembo, dembo_peres_rosen_zeitouni}. Loosely, the picture which is expected in all such models is that of "fractal-like clusters well separated from each other". \\

The key ingredient in the proof of Theorem \ref{teor main} is an identification of the extremal process of BBM with an auxiliary process constructed from a Poisson process,with an explicit density of points in the tail. 
This is a rigorous implementation of the {\it cavity approach} developed in the study of mean field spin glasses \cite{parisi} for the case of BBM, and might be of interest to determine extreme value statistics for other processes. We discuss the idea of the proof in the next section. 

\subsection{The Laplace transform of the extremal process of BBM}
Recently,  Brunet and Derrida \cite{brunet_derrida_second} have shown the existence of statistics of the extremal process ${\mathcal E}_t$ in the limit of large times.
We prove here that the limit of ${\mathcal E}_t$ exists as a point process using the convergence of the Laplace functionals,
\begin{equation}
\label{eqn: laplace}
\Psi_t(\phi)\equiv\E\left[\exp\left(-\int \phi(y) {\mathcal E}_t(dy)\right)\right],
\end{equation}
for $\phi\in \mathcal{C}_c(\R)$ non-negative.  
This gives the existence result in the proof of the main theorem.

\begin{prop}[Existence of the limiting extremal process] 
\label{teor_existence}
The point process ${\mathcal E}_t$ converges in law to a point process ${\mathcal E}$.  
\end{prop}
It is easy to see that the Laplace functional is a solution of the F-KPP equation following the observation of McKean, see Lemma \ref{mckean_fundamental} below. 
However, convergence is more subtle. It will follow from the convergence theorem of Bramson, see Theorem \ref{bramson_fundamental_convergence} below, but only after
an appropriate truncation of the functional needed to satisfy the hypotheses of the theorem.
The proof recovers a representation of the form \eqref{gumbel_like}.
More importantly, we obtain an expression for the constant $C$ as a function of the initial condition.
This observation is inspired by the work of Chauvin and Rouault \cite{chauvin_rouault}.
It will be at the heart of the representation theorem of the extremal process as a cluster process.
\begin{prop}\label{eqn: laplace rep}
Let  ${\mathcal E}_t$ be the process \eqref{extremal_process_def}.
For $\phi\in\mathcal{C}_c(\R)$ non-negative, 
\begin{equation} \label{laplace_repr_formula}
\lim_{t\to \infty} \E \left[ \exp\left(-\int \phi(y+x) {\mathcal E}_t(dy)\right)\right]= \E\left[\exp\left( - C(\phi) Z \ee^{-\sqrt{2} x}  \right)  \right]
\end{equation}
where, for $u(t,y)$ solution of F-KPP with initial condition $u(0,y)=e^{-\phi(y)}$,
$$
C(\phi)=\lim_{t\to\infty}\sqrt{\frac{2}{\pi}}\int_0^\infty \Big(1-u(t,y+\sqrt{2}t) \Big)ye^{\sqrt{2}y}dy
$$
is a strictly positive constant depending on $\phi$ only, and $Z$ is the derivative martingale. 
\end{prop}
A straightforward consequence of Proposition \ref{eqn: laplace rep}
is the {\it invariance under superpositions} of the random measure ${\mathcal E}$, conjectured by Brunet and Derrida \cite{brunet_derrida_first, brunet_derrida_second}.

\begin{cor}[Invariance under superposition] \label{invariance_bbm} 
The law of the extremal process of the superposition of $n$ independent BBM started in $x_1, \dots, x_n \in \R$ coincides in the limit of large time with that of a single BBM, up to a random shift. 
\end{cor}
As conjectured in \cite{brunet_derrida_second}, the superposability property is satisfied by any cluster point process constructed from 
a Poisson point process with exponential density to which, at each atom, is attached an iid point process. Theorem \ref{main_theorem} shows that this is indeed the case for the extremal process of branching Brownian motion, and provides a characterization of the law of the clusters in terms of branching Brownian motions conditioned on performing "unusually high" jumps. 

\subsection{An auxiliary point process}

Let $(\Omega', \mathcal F', P)$ be a probability space, and $Z: \Omega' \to \R_+$
with distribution as that of the limiting derivative martingale \eqref{eqn: deriv}. 
(Expectation with respect to $P$ will be denoted by 
$E$). 
Let $(\eta_i; i \in \N)$ be the atoms of a Poisson point process on $(-\infty,0)$ with density 
\beq \label{density}
\sqrt{\frac{2}{\pi}} \left(-x \ee^{-\sqrt{2} x} \right) \dd x \ .
\eeq 
For each $i\in \N$, consider independent branching Brownian motions with drift $-\sqrt{2}$, i.e. $\left\{x_k^{(i)}(t) -\sqrt{2} t; k \leq n_i(t)\right\}$, issued on $(\Omega', \mathcal F', P)$. 
Remark that, by \eqref{travelling_one} and \eqref{centering_kpp}, to given $i\in \N$,
\beq \label{drifting_off}
 \max_{k \leq n_i(t)} x^{(i)}_k(t)- \sqrt{2} t \to -\infty, \ \text{ $P$-a.s.} 
\eeq
The auxiliary point process of interest is the superposition of the iid BBM's with drift and shifted by $\eta_i+\frac{1}{\sqrt{2}}\log Z$:
\beq
\Pi_t \defi \sum_{i,k}\delta_{ \frac{1}{\sqrt{2}}\log Z+\eta_i+x_k^{(i)}(t)-\sqrt{2} t }
\eeq 
The existence and non-triviality of the process in the limit $t\to\infty$ is not straightforward, especially in view of  \eqref{drifting_off}.
It will be proved by recasting the problem into the frame of convergence of solutions of the F-KPP equations to travelling waves, as in the proof of Theorem \ref{teor_existence}.
It turns out that the density of the Poisson process points growing faster than exponentially as $x\to -\infty$ compensates for the fact that BBM's with drift wander off to $-\infty$. 

\begin{teor}[The auxiliary point process] \label{main_theorem}
Let  ${\mathcal E}_t$ be the extremal process \eqref{extremal_process_def} of BBM. Then
\[
\lim_{t\to \infty} {\mathcal E}_t \law \lim_{t\to \infty} \Pi_t \ .
\]
\end{teor}
The above will follow from the fact that the Laplace functionals of $\lim_{t\to\infty}\Pi_t$ admits a representation of
the form \eqref{laplace_repr_formula}, and that the constants $C(\phi)$ in fact correspond.
\begin{rem}
An elementary consequence of the above identification is that the extremal process $\mathcal{E}$ shifted back by $\frac{1}{\sqrt{2}}\log Z$ 
is an infinitely divisible point process. 
The reader is referred to \cite{kallenberg} for definitions and properties of such processes..
\end{rem}

We conjectured Theorem \ref{main_theorem} in a recent paper \cite{abk_poissonian}, where it is pointed out that such a representation 
is a natural consequence of the results on the genealogies and the paths of the extremal particles in \cite{abk}.
The proof of Theorem \ref{main_theorem} provided here does not rely on such techniques.
It is based on the analysis of Bramson \cite{bramson_monograph} and the subsequent works of Chauvin-Rouault \cite{chauvin_rouault, chauvin_rouault2}, and  Lalley and Sellke \cite{lalley_sellke}. 
However, as discussed in Section \ref{heuristics}, the results on the genealogies of \cite{abk} provides a useful heuristics.
It is  likely that such path techniques be an alternative approach to the Feynman-Kac representation on which the results of \cite{bramson_monograph} (and thus those of our present paper) are based.

We now list some of the properties of the Poisson cluster process in the limit of large times (which by Theorem \ref{main_theorem} coincide with those of the extremal process of BBM).
\begin{prop}[Poissonian nature of the cluster-extrema] \label{poissonianity_cluster_e} 
Consider $\Pi_t^{\text{ext}}$ the point process obtained by retaining from $\Pi_t$ the maximal particles of the BBM's, 
\[
\Pi_{t}^{\text{ext}} \defi \sum_i \delta_{ \frac{1}{\sqrt{2}}\log Z+\eta_i + \max_k \{ x_k^{(i)}(t) - \sqrt{2} t\}}\ .
\]
Then $\lim_{t\to \infty} \Pi_{t}^{\text{ext}} \law \emph{PPP}\left(  Z \sqrt{2} C \ee^{-\sqrt{2} x} \dd x \right)$, where $C$ is the same constant appearing in \eqref{gumbel_like}.
In particular, the maximum of $\lim_{t\to \infty} \Pi_{t}^{\text{ext}} $ has the same law as the limit law of the maximum of BBM.
\end{prop}
The fact that the laws of the maximum of the cluster-extrema and of BBM correspond is a consequence of \eqref{gumbel_like}
and the formula for the maximum of a Poisson process. 

The Poissonian nature of the cluster-extrema in the case of BBM was first proved in \cite{abk_poissonian}. 
Given the equivalence of the extremal and cluster processes, the above thus comes as no surprise. 
We will provide a different proof from the one given in \cite{abk_poissonian},
which will be useful  for the proof of Theorem \ref{main_theorem}. 

The last ingredient of the proof of Theorem \ref{teor main} is a characterization of the law of the clusters, that is the distribution of the points that are extremal and
that come from the same atom of the Poisson process $\eta$.
To this aim, it is necessary to understand which atoms in fact contribute to the extremal process.
The following result is a good control of the location of such atoms that
implies that the branching Brownian motion forming the clusters must perform unusually high jumps, of the order of $\sqrt{2}t+a\sqrt{t}$ for some $a>0$.

\begin{prop}\label{squareroot_prop}
Let $y\in \R$ and $\vare>0$ be given. There exists $0 < C_1 < C_2 < \infty$ and $t_0$ depending only on  $y$ and $\vare$,  such that
\[
\sup_{t\geq t_0} P\left[\exists_{i,k}: \eta_i + x_k^{(i)}(t)-\sqrt{2} t  \geq 
y,\, \and \: \eta_i \notin \left[-C_1 \sqrt{t}, -C_2 \sqrt{t}\right]  \right] < \vare.
\]
\end{prop}
It will be shown that the conditional law of a BBM given the event that the maximum makes a displacement greater than $\sqrt{2}t$ exists
in the limit $t\to\infty$ and, perhaps somewhat surprisingly, does not depend on the displacement.
This will entail that, seen from the cluster-extrema, the laws of 
the clusters are identical.
\begin{prop} 
\label{iid}
Let $x\defi -a\sqrt{t}+b$ for some $a>0, b\in \R$. In the limit $t\to\infty$, the conditional law of 
$\sum_{k\leq n(t)} \delta_{x+x_k(t)-\sqrt{2}t}$, given the event 
$\{ x+ \max_k x_k(t)-\sqrt{2}t  >0\}$, exists and does not depend on $x$. 
Moreover, $x+ \max_k x_k(t)-\sqrt{2}t$ conditionally on the event $\{x+ \max_k x_k(t)-\sqrt{2}t  >0\}$ converges weakly to an exponential random variable. 
\end{prop}

\subsection{A picture behind the theorem} \label{heuristics}
Our understanding of BBM stems from results on the genealogies of extremal particles that were obtained in \cite{abk}. 
BBM can be seen as a Gaussian field of correlated random variables on the configuration space $\Sigma_t \defi \{1, \dots, n(t)\}$.
Conditionally upon a realization 
of the branching, the correlations among particles are given by the {\it genealogical distance}
\beq
Q_{ij}(t) \defi \sup\{s \leq t: \; x_i(s) = x_j(s) \}, \quad i,j \in \Sigma_t.
\eeq
The genealogical distance encodes all information on correlations, and in 
particular among extremal particles.
As a first step towards a characterization of the correlation structure at the edge, we derived a characterization of the {\it paths} of extremal particles 
by identifying a mechanism of "entropic repulsion", which we shall recall. 

Let $\alpha \in (0, 1/2)$. The {\it entropic envelope} is the curve 
\beq \label{entropic}
E_{\alpha, t}(s) \defi \frac{s}{t} m(t) - e_{\alpha, t}(s),
\eeq
with 
\beq 
e_{\alpha, t}(s) \defi \begin{cases}
			s^\alpha, & \; 0\leq s \leq t/2, \\
			(t-s)^\alpha, & t/2 \leq s \leq t.
                       \end{cases}
\eeq
For $D\subset \R$, we denote by $\Sigma_t(D)\defi \{k\in \Sigma_t:\; x_k(t) - m(t) \in D \}$ the set of particles 
falling at time $t$ into the subset $D+m(t)$. A first result reads: 
\begin{teor}[\cite{abk}] \label{entropic_paths} For compact $D\subset \R$, 
\[
\lim_{r_d, r_g \uparrow \infty} \lim_{t\uparrow \infty} \PP\left[ \exists k\in \Sigma_t(D) \; \text{such that} \; x_k(s) \geq E_{\alpha, t}(s) \; \text{for some } s\in (r_d, t-r_g) \right] = 0. 
\]
\end{teor}
The content of the above is that extremal particles lie typically well below the maximal displacement, namely during the interval 
$(r_d, t-r_g)$. Remark that $r_d, r_g = o(t)$ as $t\to\infty$.

Proposition \ref{entropic_paths} allows to considerably restrict the correlations of particles at the edge. 
\begin{teor}[\cite{abk}] \label{overlaps} For compact $D\subset \R$,
\[
\lim_{r_d, r_g \uparrow \infty} \lim_{t\uparrow \infty} \PP\left[ \exists j, k\in \Sigma_t(D) \; \text{such that} \; Q_{jk}(t) \in (r_d, t-r_g) \right] = 0. 
\]
\end{teor}
According to this theorem, extremal particles can be split into groups of "weakly dependent" random variables (those particles with most recent 
commen ancestors in $[0,r_d]$, 
or in groups of "heavily dependent" random variables (those particles with
 most recent common ancestors  in $[t-r_g, t]$). The image which emerges from Theorems 
\ref{entropic_paths} and \ref{overlaps} is depicted in Figure $1$ below. 

\begin{figure}[ht] \label{fig.1}
\begin{center}
  \input{picture_two.pstex_t}
\end{center}
\caption{Genealogies of extremal particles 
}
\end{figure}
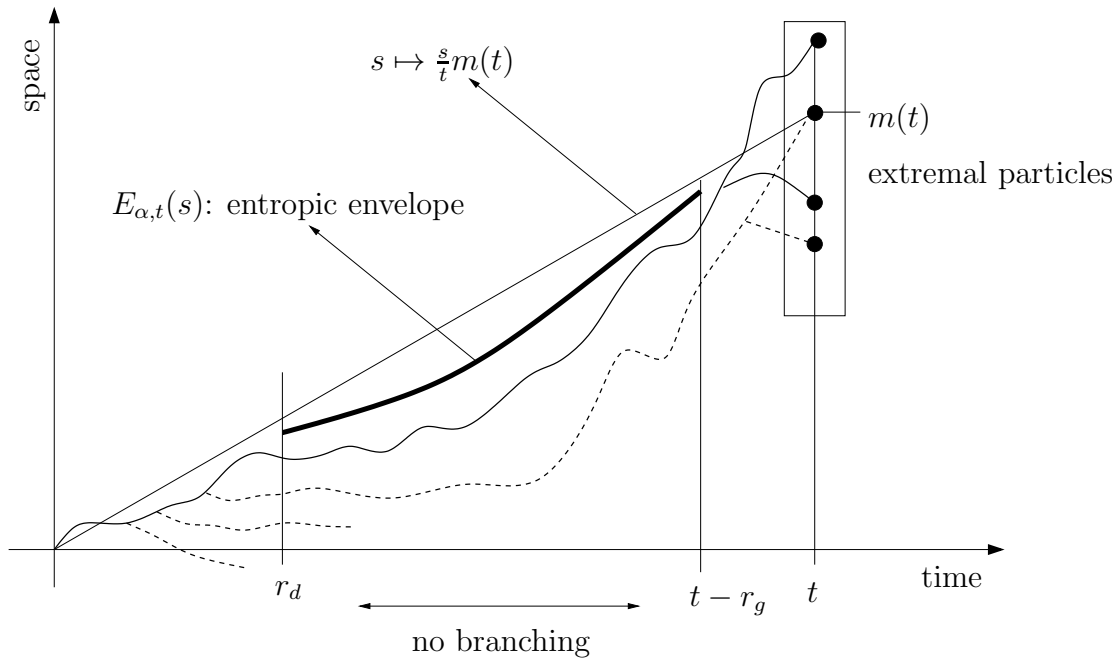

Hence, branching happens at the very beginning, after which particles continue along {\it independent} paths, and start branching again only towards the end.
It is not difficult to see \cite{lalley_sellke, abk_poissonian}  that the branching at the beginning is responsible for the appearance of the derivative martingale in the large time limit.
On the other hand, the branching towards the end creates the clusters. 
By Theorem \ref{entropic_paths} these are branching Brownian motions of length $r_g$ which have to perform displacements of order (at least) $\sqrt{2} r_g$; 
in fact, the displacements must be even bigger, in order to compensate for the "low" heights of the ancestors at time $t-r_g$.
According to the picture, and conditionally on what happened up to time $r_d$ (the initial branching) particles at time 
$t-r_g$ which have offsprings in the lead at the future time $t$ have different ancestors at time $r_d$. Thus, one may expect that the point process of such ancestors   should be ("close to") a {\it Poisson point process}; as such, only its density should matter.  It is not difficult to see that the density of particles at time $t-r_g$ whose paths remain below the entropic envelope during all of the interval $(r_d, t-r_g)$
is in first approximation (and conditionally on the initial branching) of the form $- x \ee^{-\sqrt{2} x} \dd x$, for $x$ large on the negative, in agreement
with Theorem \ref{main_theorem}. This heuristics stands behind the existence of the Poisson cluster process representation.  

The phenomenon described above is however more delicate than it might appear at first reading. This is in particular due to the following intriguing fact. Namely, the feature whether a particle alive at time $t-r_g$ has an offspring which makes it to the lead at the future time $t$ is evidently not measurable w.r.t. the $\sigma$-field generated up to time $t-r_g$. In other words: {\it at any given time, leaders are offsprings of ancestors which are chosen 
according to whether their offsprings are leaders at the considered time!} We shall provide below yet another point of view on the issue which again suggests that such an intricate random thinning  based on the future evolution
eventually leads to the existence of the Poisson cluster process representation given in Theorem \ref{main_theorem}. 

As it turns out, the Poisson cluster representation is of relevance for the study of spin glasses \cite{Bo05, talagrand}, in particular within the frame of competing particle systems [CPS], a wide-ranging set of ideas which address the statistical properties of extremal processes, as pursued by Aizenman and co-authors \cite{aizenman_arguin,  aizenman_sims_starr2, aizenman_ruzmaikina}. The CPS-approach can be seen as an attempt to formalize to so-called {\it cavity method} developed by Parisi and co-authors \cite{parisi} for the study of spin glasses. The connection with Theorem \ref{main_theorem} is as follows.  

Let ${\mathcal E} =\sum_{i\in\N}\delta_{e_i}$ be the limiting extremal process of a BBM starting at zero. 
It is clear since $m(t) = m(t-s) + \sqrt{2} s + o(1)$ that the law of ${\mathcal E}$ satisfies the following invariance property.
For any $s\geq 0$,
\beq \label{cavity_invariance_bbm}
{\mathcal E} \law \sum_{i,k} \delta_{e_i+ x_k^{(i)}(s)-\sqrt{2} s}
\eeq
where $\{x_k^{(i)}(s); k\leq n^{(i)}(s)\}_{i\in \N}$ are iid BBM's.
 On the other hand, by Theorem \ref{main_theorem}, 
\beq \label{cavity_step_one}
{\mathcal E} \law \lim_{s\to \infty}  \sum_{i,k} \delta_{\eta_i + x_k^{(i)}(s) - \sqrt{2}s }\ ,
\eeq
where now $(\eta_i; i\in \N)$ are the atoms of a PPP$\left(\sqrt{\frac{2}{\pi}}-x \ee^{-\sqrt{2}x}dx\right)$ shifted by $\frac{1}{\sqrt{2}} \log Z$. 

The main idea of the CPS-approach is to characterize extremal processes as invariant measures under suitable, model-dependent stochastic mappings. 
In the case of branching Brownian motion, the mapping consists of adding to each ancestors $e_i$ independent BBMs with drift $-\sqrt{2}$. 
This procedure randomly picks of the original ancestors only those with offsprings in the lead at some future time. But the random 
thinning is performed through independent random variables: this suggests that one may  indeed replace the process of ancestors $\{e_i\}$ by 
a Poissonian process with suitable density. 

Behind the random thinning, a crucial phenomenon of "energy vs. entropy" is at work. Under the light of \eqref{drifting_off}, the probability that any such BBM with drift $-\sqrt{2}$ attached to the process of ancestors does not wander off to $-\infty$ vanishes in the limit of large times. On the other hand, the higher the position of the ancestors, the fewer one finds. A delicate balance must therefore be met, and only ancestors lying on a precise level below the lead can survive the random thinning. This is indeed the content of Proposition \ref{squareroot_prop}, and a fundamental ingredient in the CPS-heuristics: particles at the edge come from the tail in the past. A key step in identifying the equilibrium measure is thus to identify the tail. In the example of BBM, we are able to show, perhaps indirectly, that a good approximation of the tail is a Poisson process with density $ -x \ee^{-\sqrt{2}x}dx$ (up to constant). \\

To make  the above heuristics rigorous we rely on the analytic approach 
pioneered by Bramson, which highlights once more the power of the connection
between BBM and the F-KPP equations. It may still be interesting to have 
a purely probabilistic proof, exploiting the detailed analysis of the
path-properties of extremals particles. This represents, however, major technical challenges. \\

\section{Proofs} \label{proofs}
In what follows, $\{x_k(t),\, k\leq n(t)\}$ will always denote a branching Brownian motion of length $t$ started in zero. 

\subsection{Technical tools}
We start by stating two fundamental results that will be used extensively. 
First, McKean's insightful observation:
 
\begin{lem}[\cite{mckean}] \label{mckean_fundamental} Let $f: \R \to [0,1]$ and $\{x_k(t): k\leq n(t)\}$ a branching Brownian motion starting at $0$.
The function
\[
u(t,x) = \E\left[ \prod_{k=1}^{n(t)} f(x+x_k(t))\right]
\]
is solution of the F-KPP equation \eqref{kpp_equation} with $u(0,x)= f(x)$. 
\end{lem}

Second, the fundamental result by Bramson on the convergence of solutions of the F-KPP equation to travelling waves: 

\begin{teor}[Theorem A \cite{bramson_monograph}] \label{bramson_fundamental_convergence}
 Let $u$ be solution of the F-KPP equation \eqref{kpp_equation} with $0\leq u(0,x) \leq 1$. Then 
\beq \label{kpp_general_form}
u(t, x+m(t)) \to \omega(x), \quad \text{uniformly in} \; x\; \text{as} \; t\to \infty,
\eeq
where $\omega$ is the unique solution (up to translation) of
\[
\frac{1}{2} w'' + \sqrt{2} \omega' + \omega^2 - \omega = 0,
\]
if and only if 
\begin{itemize}
\item[1.] for some $h>0$, $\limsup_{t\to \infty} \frac{1}{t} \log \int_t^{t(1+h)} (1-u(0,y)) \dd y \leq -\sqrt{2}$;
\item[2.] and for some $\nu>0$, $M>0$, $N>0$, $\int_{x}^{x+N} (1-u(0,y)) \dd y > \nu$ for all $x\leq - M$. 
\end{itemize}
Moreover, if $\lim_{x\to \infty} \ee^{b x}(1-u(0,x)) = 0$ for some $b> \sqrt{2}$, then one may choose
\beq \label{choice_m} 
m(t) = \sqrt{2} t -\frac{3}{2\sqrt{2}} \log t.
\eeq
\end{teor}
It is to be noted that the necessary and sufficient conditions hold for uniform convergence in $x$. 
Pointwise convergence could hold when, for example, condition 2 is not satisfied This is the case in Theorem \ref{teor_existence}.

It will be often convenient to consider the F-KPP equation in the following form: 
\beq \label{kpp_one_minus}
u_t = \frac{1}{2} u_{xx} + u - \sum_{k=1}^\infty p_k u^k.
\eeq
The solutions to \eqref{kpp_equation} and \eqref{kpp_one_minus} are simply related via the transformation $u \to 1-u$.
The following Proposition provides sharp approximations to the solutions of such equations.

\begin{prop}
\label{bramson_fundamental_interpolation}
Let $u$ be a solution to the F-KPP equation \eqref{kpp_one_minus} with initial data satisfying
$$
\int_0^\infty y \ee^{\sqrt{2}y}u(0,y)dy<\infty\ ,
$$
and such that $u(t,\cdot+m(t))$ converges.  Define
\beq \label{psi_no_m}
\psi(r,t,X+\sqrt{2t})\defi \frac{e^{-\sqrt{2} X}}{\sqrt{2\pi(t-r)}}\int_0^\infty u(r,y'+\sqrt{2}r)e^{y'\sqrt{2}}e^{-\frac{(y'-X)^2}{2(t-r)}}(1-e^{-2y'\frac{(X+\frac{3}{2\sqrt{2}}\log t)}{t-r}})dy'
\eeq 
Then for $r$ large enough, $t\geq 8r$, and $X\geq 8r - \frac{3}{2\sqrt{2}}\log t$, 
\beq \label{upper_lower_approx}
\gamma^{-1}(r) \psi(r,t,X+\sqrt{2t}) \leq u(t, X+\sqrt{2} t) \leq \gamma(r) \psi(r,t,X+\sqrt{2t}),
\eeq
for some $\gamma(r)\downarrow 1$ as $r\to \infty$. 
\end{prop}
With the notations from the above Proposition, the function $\psi$ thus fully captures the large space-time behavior of the solution to the F-KPP equations.
We will make extensive use of \eqref{upper_lower_approx}, mostly when both $X$ and $t$ are large in the positive, in which case the dependence
on $X$ becomes particularly easy to handle. 
\begin{proof}[Proof of Proposition \ref{bramson_fundamental_interpolation}]
For $T>0$ and $0<\alpha< \beta <\infty$, let $\{\mathfrak z_{\alpha, \beta}^T(s), 0\leq s\leq T\}$ 
denote a Brownian bridge of length $T$ starting in $\alpha$ and ending in $\beta$. 

It has been proved by Bramson (see \cite[Proposition 8.3]{bramson_monograph}) that for $u$ satisfying the assumptions in the Proposition \ref{bramson_fundamental_interpolation}, the following holds:
\begin{enumerate}
\item for $r$ large enough, $t\geq 8r$ and $x\geq m(t)+8r$
$$
u(t,x)\geq C_1(r) e^{t-r}\int_{-\infty}^\infty u(r,y) \frac{e^{-\frac{(x-y)^2}{2(t-r)}}}{\sqrt{2\pi(t-r)}}\PP\left[\mathfrak{z}^{t-r}_{x,y}(s)>\overline{\mathcal{M}}_{r,t}^x(t-s),s\in[0,t-r]\right]dy
$$
and
$$
u(t,x)\leq C_2(r) e^{t-r}\int_{-\infty}^\infty u(r,y) \frac{e^{-\frac{(x-y)^2}{2(t-r)}}}{\sqrt{2\pi(t-r)}}\PP\left[\mathfrak{z}^{t-r}_{x,y}(s)>\underline{\mathcal{M}}_{r,t}'(t-s),s\in[0,t-r]\right]dy
$$
where the functions $\overline{\mathcal{M}}_{r,t}^x(t-s)$, $\underline{\mathcal{M}}_{r,t}'(t-s)$ satisfy 
$$
\underline{\mathcal{M}}_{r,t}'(t-s)\leq n_{r,t}(t-s)\leq \overline{\mathcal{M}}_{r,t}^x(t-s) \ ,
$$
for $n_r(s)$ being the linear interpolation between $\sqrt{2}{r}$ at time $r$ and $m(t)$ at time $t$.
Moreover, $C_1(r)\uparrow 1$, $C_2(r)\downarrow 1$ as $r\to\infty$.
\item If $\psi_1(r,t,x)$ and $\psi_2(r,t,x)$ denote respectively the lower and upper bound to $u(t,x)$, we have
$$
1\leq \frac{\psi_2(r,t,x)}{\psi_1(r,t,x)}\leq \gamma(r)
$$
where $\gamma(r)\downarrow 1$ as $r\to\infty$.
\end{enumerate}
Hence, if we denote by 
$$
\widehat \psi(r,t,x)=e^{t-r}\int_{-\infty}^\infty u(r,y) \frac{e^{-\frac{(x-y)^2}{2(t-r)}}}{\sqrt{2\pi(t-r)}}\PP(\mathfrak{z}^{t-r}_{x,y}(s)>n_{r,t}(t-s),s\in[0,t-r])dy \ , 
$$
we have by domination $\psi_1\leq \widehat \psi\leq \psi_2$. Therefore, for $r,t$ and $x$ large enough
\begin{equation}
\label{eqn: bound1}
\frac{u(t,x)}{\widehat \psi(r,t,x)}\leq \frac{\psi_2(r,t,x)}{\widehat \psi(r,t,x)}\leq \frac{\psi_2(r,t,x)}{\psi_1(r,t,x)}\leq \gamma(r)\ , 
\end{equation}
and
\begin{equation}
\label{eqn: bound2}
\frac{u(t,x)}{\widehat \psi(r,t,x)}\geq \frac{1}{\gamma(r)}\ .
\end{equation}
Combining \eqref{eqn: bound1} and \eqref{eqn: bound2} we thus get
\beq \label{sequence_of_less}
\gamma^{-1}(r) \widehat \psi(r,t,x) \leq u(t,x) \leq \gamma(r) \widehat \psi(r,t,x).
\eeq
We now consider $X \geq 8r - \frac{3}{2\sqrt{2}} \log t$, and obtain from \eqref{sequence_of_less} that
\beq \label{sequence_of_less_new}
\gamma^{-1}(r) \widehat \psi(r,t,X+\sqrt{2}t) \leq u(t,X+\sqrt{2}t) \leq \gamma(r) \widehat \psi(r,t,X+\sqrt{2}t).
\eeq
The probability involving the Brownian bridge in the definition of $\widehat \psi$ can be explicitly computed.
The probability of a Brownian bridge of length $t$ to remain below the interpolation of $A>0$ at time $0$ and $B>0$ at time $t$
is $1-e^{-2AB/t}$, see e.g. \cite{scheike}. In the above setting the length is $t-r$, $A= \sqrt{2}t+x-m(t)=x + \frac{3}{2\sqrt{2}}\log t>0$ for $t$ large enough and $B=y-\sqrt{2}r=y'$. Using this, together with the fact that $\PP(\mathfrak{z}^{t-r}_{x,y}(s)>n_{r,t}(t-s),s\in[0,t-r])$ is $0$ for $B=y'<0$, and by change of variable $y'=y+\sqrt{2}t$ in the integral appearing in the definition of $\widehat \psi$, we get 
\beq \bea
\widehat \psi(r,t,X+\sqrt{2}t) &=\frac{e^{-\sqrt{2} X}}{\sqrt{2\pi(t-r)}}\int_0^\infty u(r,y'+\sqrt{2}r)e^{y'\sqrt{2}}e^{-\frac{(y'-X)^2}{2(t-r)}}(1-e^{-2y'\frac{(X+\frac{3}{2\sqrt{2}}\log t)}{t-r}})dy' \\
& = \psi(r,t, X+\sqrt{2}t).
\eea \eeq
This, together with \eqref{sequence_of_less},
concludes the proof of the proposition. 
\end{proof}

The bounds in \eqref{upper_lower_approx} have been used by Chauvin and Rouault to compute the 
probability of deviations of the maximum of BBM, see Lemma 2 \cite{chauvin_rouault}. 
Their reasoning applies to solutions of the F-KPP equation with other initial conditions than those corresponding to the maximum.
We give the statement below, and reproduce Chauvin and Rouault's proof in a general setting for completeness.
\begin{prop} \label{cr_our_setting} 
Let the assumptions of Proposition \ref{bramson_fundamental_interpolation} be satisfied, and 
assume furthermore that $y_0= \sup\{y: u(0,y) >0 \}$ is finite. Then, 
$$
\lim_{t\to\infty }e^{x\sqrt{2}}\frac{t^{3/2}}{\log t} \psi(r,t,x+\sqrt{2t})=\frac{3}{2\sqrt{\pi}}\int_0^\infty ye^{y\sqrt{2}}u(r,y+\sqrt{2}r)dy
$$
Moreover, the limit of the right-hand side exists as $r\to\infty$, and it is positive and finite.
\end{prop}

\begin{proof}
The first claim is straightforward if we can take the limit $t\to\infty$ inside the integral in the definition of $\psi$. We need to justify this using dominated convergence. Since $e^{-x}\geq 1-x$ for $x>0$, the integrand in the definition of $\psi$ times $e^{x\sqrt{2}}\frac{t^{3/2}}{\log t}$
is smaller than 
\beq \label{to_bound_unif}
(cste) y' e^{y'\sqrt{2}}u(r,y'+\sqrt{2}r). 
\eeq
It remains to show that \eqref{to_bound_unif} is integrable in $y' \geq 0$. To see this, let 
$u^{(2)}(t,x)$ be the solution to $\partial_t  u^{(2)} = \frac{1}{2} u^{(2)}_{xx} - u^{(2)}$ [the {\it linearised} F-KPP-equation \eqref{kpp_one_minus}] with initial conditions $u^{(2)}(0,x) = u(0,x)$. By the maximum principle for nonlinear, parabolic pde's, see e.g. \cite[Corollary 1, p.29]{bramson_monograph}, 
\beq \label{dominated_one}
u(t,x)\leq u^{(2)}(t,x)
\eeq 
Moreover, by the Feynman-Kac representation and the definition of $y_0$, 
\beq \label{dominated_two}
u^{(2)}(t,x)\leq e^t \int_{-\infty}^{y^0} \frac{1}{\sqrt{2\pi t}}e^{-\frac{(y-x)^2}{2t}}dy,
\eeq
and for any $x>0$ we thus have the bound
\beq \label{dominated_three}
u^{(2)}(t,x) \leq e^t e^{-\frac{x^2}{2t}} e^{\frac{y_0 x}{t}}\ .
\eeq
Hence, 
\beq \label{dominated_four}
u(r,y+\sqrt{2} r)\leq e^{-\frac{y^2}{2r}}e^{\frac{y_0y}{2r}}e^{-y\sqrt{2}}\ .
\eeq
The upper bound is integrable over the desired measure since
\beq \label{dominated_five}
\int_0^\infty y e^{\sqrt{2} y} e^{-\frac{y^2}{2r}}e^{\frac{y_0y}{2r}}e^{-y\sqrt{2}} dy=\int_0^\infty y  e^{-\frac{y^2}{2r}} e^{\frac{y_0y}{2r}} dy <\infty \ .
\eeq
Therefore dominated convergence can be applied and the first part of the Proposition follows.

It remains to show that 
$$
\lim_{r\to\infty}\int_0^\infty ye^{y\sqrt{2}}u(r,y+\sqrt{2}r)dy \text{ exists and is finite.}
$$
Write $C(r)$ for the integral. By Proposition \ref{bramson_fundamental_interpolation},  for $r$ large enough,
$$
 \limsup_{t\to\infty }e^{x\sqrt{2}}\frac{t^{3/2}}{\log t} u(t,x+\sqrt{2t})\leq \gamma(r)\lim_{t\to\infty}e^{x\sqrt{2}}\frac{t^{3/2}}{\log t} \psi(r,t,x+\sqrt{2t})=C(r)\gamma(r)\ ,
$$
and
$$
\liminf_{t\to\infty }e^{x\sqrt{2}}\frac{t^{3/2}}{\log t} u(t,x+\sqrt{2t})\geq \gamma(r)^{-1}\lim_{t\to\infty}e^{x\sqrt{2}}\frac{t^{3/2}}{\log t} \psi(r,t,x+\sqrt{2t})=C(r)\gamma(r)^{-1}\ ,
$$
Therefore since $\gamma(r)\to 1$
$$
\limsup_{t\to\infty }e^{x\sqrt{2}}\frac{t^{3/2}}{\log t} u(t,x+\sqrt{2t})\leq \liminf_{r\to\infty}C(r)
$$
and
$$
\liminf_{t\to\infty }e^{x\sqrt{2}}\frac{t^{3/2}}{\log t} u(t,x+\sqrt{2t})\geq \limsup_{r\to\infty}C(r)\ .
$$
It follows that $\lim_{r\to\infty} C(r)=:C$ exists and so does $\lim_{t\to\infty }e^{x\sqrt{2}}\frac{t^{3/2}}{\log t} u(t,x+\sqrt{2t})$. Moreover $C>0$ otherwise 
\beq
\lim_{t\to\infty }e^{x\sqrt{2}}\frac{t^{3/2}}{\log t} u(t,x+\sqrt{2t})=0
\eeq
which is impossible since $$\lim_{t\to\infty }e^{x\sqrt{2}}\frac{t^{3/2}}{\log t} u(t,x+\sqrt{2t})\geq C(r)/\gamma(r)$$ for $r$ large enough but finite
($\gamma(r)$ and $C(r)$ are finite for $r$ finite).
Moreover $C<\infty$, otherwise $$\lim_{t\to\infty }e^{x\sqrt{2}}\frac{t^{3/2}}{\log t} u(t,x+\sqrt{2t})=\infty,$$ which is impossible
since $\lim_{t\to\infty }e^{x\sqrt{2}}\frac{t^{3/2}}{\log t} u(t,x+\sqrt{2t})\leq C(r)\gamma(r)$ for $r$ large enough, but finite.
\end{proof}

\subsection{Existence of a limiting process}

\begin{proof}[Proof of Proposition \ref{teor_existence}] 
It suffices to show that, for $\phi \in \mathcal C_c(\R)$ positive, the Laplace transform $\Psi_t(\phi)$, defined in \eqref{eqn: laplace}, of the extremal process of branching Brownian motion 
converges. 

Remark first that this limit cannot be $0$, since in the case of BBM it can be checked \cite{abk} that
\[
\lim_{N \to \infty} \lim_{t \to \infty} \PP\left[{\mathcal E}_t(B) > N \right] = 0, \; \text{for any bounded measurable set}\; B\subset \R,
\]
hence the limiting point process must be locally finite. 

For convenience, we define $\max \mathcal{E}_t \defi \max_{k\leq n(t)} x_k(t)-m(t)$.
By Theorem \ref{bramson_fundamental_convergence} applied to the function 
\[
u(t, \delta+m(t)) = \E\left[\prod_{k=1}^{n(t)}\1_{\left\{x_k(t) -m(t)\leq \delta \right\}}\right]=\PP[\max \mathcal{E}_t\leq\delta]
\]
it holds that 
\beq
\label{eqn: exist cutoff}
\lim_{\delta\to\infty}\lim_{t\to\infty} 1-u(t, \delta+m(t))=\lim_{\delta\to\infty} 1-\omega(\delta)=0\ .
\eeq

Now consider for $\delta>0$
\beq\label{split_laplace}\bea
\E\left[\exp\left( -\int \phi ~ d\mathcal{E}_t  \right)\right] & = 
\E\left[\exp\left( -\int \phi ~ d\mathcal{E}_t  \right) 
\1_{\{\max \mathcal{E}_t \leq \delta\}}\right] \\
& \qquad + \E\left[\exp\left( -\int \phi ~d\mathcal{E}_t  \right) \1_{\{\max \mathcal{E}_t > \delta\}}\right]
\eea \eeq
Note that by \eqref{eqn: exist cutoff}, the second term on the r.h.s of \eqref{split_laplace} satisfies 
$$
\limsup_{\delta\to\infty}\limsup_{t\to\infty}\E\left[\exp\left( -\int \phi ~d\mathcal{E}_t  \right) \1_{\{\max \mathcal{E}_t > \delta\}}\right]
\leq \limsup_{\delta\to\infty}\limsup_{t\to\infty}\PP[\max \mathcal{E}_t>\delta]=0\ .
$$
It remains to address the first term on the r.h.s of \eqref{split_laplace}.
Write for convenience 
$$
\Psi_t^{\de}(\phi) \defi \E\left[\exp\left( -\int \phi ~ d\mathcal{E}_t  \right) 
\1_{\{\max \mathcal{E}_t \leq \delta\}}\right]\ .
$$
We claim that the limit
\beq
\lim_{\delta\to \infty} \lim_{t\to \infty} \Psi_t^{\de}(\phi) \defi \Psi(\phi)
\eeq 
exists, and is strictly smaller than one. To see this, set $$g_\de(x) \defi \ee^{-\phi(x)} \1_{\{x \leq \delta\}}\ ,$$ and 
\beq 
u_\de(t, x) \defi \E\left[ \prod_{k\leq n(t)} g_\de(-x+x_k(t))\right].
\eeq
By Lemma \ref{mckean_fundamental}, $u_\de$ is then solution to the F-KPP equation with $u_\de(0,x)= g_\de(-x)$. Moreover, $g_\de(-x)= 1$ for $x$ large enough in the positive, and $g_\de(-x)=0$ for $x$ large enough in the negative, so that conditions $(1)$ and $(2)$ of Theorem \ref{bramson_fundamental_convergence} are satisfied as well as \eqref{choice_m} on the form of $m(t)$. 
Note that this would not be the case without the presence of the cutoff. Therefore 
\beq 
u_\de(t, x+m(t)) = \E\left[\prod_{k=1}^{n(t)} g_\de(-x+x_k(t)-m(t))\right]
\eeq
converges as $t\to\infty$ uniformly in $x$ by Theorem \ref{bramson_fundamental_convergence}. But 
\beq \bea 
\Psi_t^\de(\phi) & = \E\left[\exp\left( -\int \phi ~ d{\mathcal E}_t  \right)\1_{\{\max \mathcal{E}_t \leq \delta\}}\right] \\
& = \E\left[ \prod_{k\leq n(t)} \exp\Big( - \phi\big(x_k(t)-m(t)\big)  \Big)\1_{\{x_k(t)-m(t) \leq \delta\}}\right] \\
& = \E\left[ \prod_{k\leq n(t)} g_\de(x_k(t)-m(t)) \right] = u_\de(t, 0+m(t)),
\eea\eeq
and therefore the limit $\lim_{t\to \infty} \Psi_t^\de(\phi) \defi \Psi^\de(\phi)$ exists. But the function $\de \mapsto \Psi^\de(\phi)$ is increasing 
and smaller than one, by construction. Therefore, $\lim_{\de\to \infty} \Psi^\de(\phi) = \Psi(\phi)$ exists. Moreover, nonnegativity of $\phi$ 
implies that $\Psi_t^\de(\phi) \leq \PP[\max \mathcal{E}_t\leq\delta]$: taking the limit $t\to \infty$ first and $\delta \to \infty$ next thus shows that $\Psi(\phi) < 1$, and concludes the proof.
\end{proof}

\subsection{The process of cluster-extrema} 
Proposition \ref{cr_our_setting} can be used to obtain an elementary proof 
of the convergence of the process of the cluster-extrema towards the PPP$\left( C Z \sqrt{2}\ee^{-\sqrt{2}x} \dd x \right)$. 
We start by proving two lemmas that will be of use later on.
\begin{lem}
Let $u(t,x)$ be a solution to F-KPP with initial condition $u(0,x)$ satisfying the assumption of Proposition \ref{cr_our_setting}. 
Let
$$
C\defi \lim_{r\to\infty}\sqrt{\frac{2}{\pi}}\int_0^\infty y e^{y\sqrt{2}} u(t,y+\sqrt{2}r) \dd y\ ,
$$
then for any $x\in\R$:
\beq 
\label{eqn: shift const}
\lim_{r\to\infty}\sqrt{\frac{2}{\pi}}\int_0^\infty y e^{y\sqrt{2}}u(t,x+y+\sqrt{2}t)=C e^{-\sqrt{2}x}\ .
\eeq
\end{lem}
\begin{proof}
By Proposition \ref{cr_our_setting}
\begin{equation}
\label{eqn: shift}
\begin{aligned}
\lim_{r\to\infty}\int_0^\infty y e^{y\sqrt{2}}u(t,x+y+\sqrt{2}t)&=
 e^{-x\sqrt{2}}\frac{2\sqrt{\pi}}{3}\lim_{t\to\infty}e^{x\sqrt{2}}\frac{t^{3/2}}{\log t}u(t,x+\sqrt{2}t)\\
&= e^{-x\sqrt{2}}\lim_{r\to\infty}\int_0^\infty y e^{y\sqrt{2}} u(t,y+\sqrt{2}r)\\
\end{aligned}
\end{equation}
\end{proof}
A straightforward consequence of the lemma, taking $u(0,x)=\1_{\{x>a\}}$, $a\in\R$, is the following vague convergence of the maximum when integrated over the
appropriate density.
\begin{lem}
\label{lem: max convergence}
For any continuous function $h:\R\to\R$ that is bounded at $+\infty$ and vanishes at $-\infty$, 
$$
\bea
&\lim_{t\to\infty} \int_{-\infty}^0\left\{\int_\R h(x) ~ \PP\Big(\max_{i}x_i(t)-\sqrt{2}t+y \in dx\Big)\right\} \sqrt{\frac{2}{\pi}}(-y e^{-\sqrt{2} y})dy \\
&\hspace{10cm}=\int_\R h(a)  \sqrt{2}  C e^{-\sqrt{2}a}da\ .
\eea
$$
\end{lem}

\begin{proof}[Proof of Proposition \ref{poissonianity_cluster_e}]
Consider
$$
E \exp -\sum_i \phi(\eta_i+M^i(t)-\sqrt{2}t)
$$
where $\eta=(\eta_i)$ is Poisson with density $\sqrt{\frac{2}{\pi}}(-xe^{-\sqrt{2}x})dx$ on $(-\infty,0)$ and
$M^i(t)\defi\max_k x_k^{(i)}(t)$.
We show that
\begin{equation}
\label{eqn: to show}
\lim_{t\to\infty}E \exp -\sum_i \phi(\eta_i+M^i(t)-\sqrt{2}t)=\exp -C\int_\R (1-e^{-\phi(x)})e^{-\sqrt{2}x}dx\ .
\end{equation}
Since the underlying process is Poisson and the $M^i$'s are iid,
$$
E \exp- \sum_i \phi(\eta_i+M^i(t)-\sqrt{2}t)=\exp -\int_{-\infty}^0\left(1-E \left[e^{-\phi(x+M(t)-\sqrt{2t}}\right]\right)\sqrt{\frac{2}{\pi}}(-xe^{-\sqrt{2}x})dx\ .
$$
The result then follows from Lemma \ref{lem: max convergence} after taking the limit.
\end{proof}

\subsection{The Laplace functional and the F-KPP equation}

\begin {proof}[Proof of Proposition \ref{eqn: laplace rep}]
The proof of the proposition will be broken into proving two lemmas.

In the first lemma 
we establish an integral representation for the 
Laplace functionals of the extremal process of BBM which are truncated by a 
certain cutoff; in the second lemma we show that the results continues to 
 holds when
 the cutoff is lifted. Throughout this section, $\phi:\R\to [0,\infty)$ is a 
non-negative continuous function with compact support. 

\begin{lem}
\label{lem: rep cutoff} 
Consider
$$
u_\delta(t,x)\defi1-\E \left[\exp\left(- \sum_k \phi(-x+x_k(t)) \right) 
1_{\left\{\max_{k\leq n(t)} -x+ x_k(t) \leq \delta\right\}} \right].
$$
Then $u_\delta(t,x)$ is the solution of the F-KPP equation \eqref{kpp_one_minus} with initial condition
$u_\delta(0,x)=1-e^{-\phi(-x)}1(-x\leq \delta)$. Moreover,
\beq \label{eqn: convergence l}
\lim_{t\to\infty} u_\delta(t,x+m(t))= 1-\E\left[ \exp -C(\phi,\delta)Ze^{-\sqrt{2}x}\right]\ ,
\eeq
where
\begin{equation}
\label{eqn: kappa rep 1}
C(\phi,\delta)= \lim_{t\to\infty} \sqrt{\frac{2}{\pi}}\int_0^{\infty} u_\delta(t,y+\sqrt{2}t)ye^{\sqrt{2}y}dy\ .
\end{equation}

\end{lem}
\begin{proof}
The first part of the Lemma is proved in the proof of Theorem \ref{teor_existence}, whereas \eqref{eqn: convergence l} follows from Theorem \ref{bramson_fundamental_convergence} and the representation \eqref{gumbel_like}. It remains to prove \eqref{eqn: kappa rep 1}. The proof is a refinement of Proposition \ref{cr_our_setting} that recovers the asymptotics \eqref{to_the_right}.

For $u_\de$ as above, let $\psi(r,t,x)$ be its approximation as in Proposition \ref{bramson_fundamental_interpolation} and choose $x, r$ so that $x\geq m(t)+8r$. 
By Proposition \ref{bramson_fundamental_interpolation} we then have the bounds
\beq \label{bounds}
\frac{1}{\gamma(r)}\psi(r, t , x+m(t)) \leq u_\delta(t, x+m(t)) \leq  \gamma(r)\psi(r, t , x+m(t))
\eeq
where
$$
\psi(r,t,x+m(t))=\frac{t^{3/2}e^{-\sqrt{2} x}}{\sqrt{2\pi(t-r)}}\int_0^\infty u_\delta(r,y'+\sqrt{2}r)e^{y'\sqrt{2}}e^{-\frac{(y'-x+\frac{3}{2\sqrt{2}\log t})^2}{2(t-r)}}(1-e^{-2\frac{y'x}{t-r}})dy' \ .
$$ 
Using dominated convergence as in Proposition \ref{cr_our_setting}, one ets
$$ 
\lim_{t\to\infty}\psi(r,t,x+m(t))=\frac{2xe^{-\sqrt{2} x}}{\sqrt{2\pi}}\int_0^\infty u_\delta(r,y'+\sqrt{2}r)y'e^{y'\sqrt{2}}dy'.
$$
Putting this back in \eqref{bounds},
\beq \label{bounds_three}
\frac{1}{\gamma(r)}C(r) \leq \lim_{t\to\infty}\frac{u_\delta(t, x+m(t))}{xe^{-\sqrt{2}x}}\leq \gamma(r)C(r),
\eeq
for
$
C(r) \defi \sqrt{\frac{2}{\pi}}\int_0^\infty u_\delta(r,y'+\sqrt{2}r)y'e^{y'\sqrt{2}}dy,
$ 
and $x> 8r$. We know that $\lim_{r\to \infty} C(r) \defi C>0$ exists by Proposition \ref{cr_our_setting}.
Thus taking $x=9r$, letting $r\to \infty$ in \eqref{bounds_three}, and using that 
$\gamma(r)\downarrow 1$, one has
$$
\lim_{x\to\infty}\lim_{t\to\infty}\frac{u_\delta(t, x+m(t))}{xe^{-\sqrt{2}x}}=
\lim_{r\to\infty}\sqrt{\frac{2}{\pi}}\int_0^\infty u_\delta(r,y'+\sqrt{2}r)y'e^{y'\sqrt{2}}dy\ .
$$
On the other hand, the representation \eqref{gumbel_like} and the asymptotics \eqref{to_the_right} yield
$$
\lim_{x\to \infty}\lim_{t\to\infty}\frac{u_\delta(t, x+m(t))}{xe^{-\sqrt{2}x}}= \frac{1-\E\left[\exp\left(- C(\phi,\delta) Z \ee^{-\sqrt{2}x} \right)\right]}{xe^{-\sqrt{2}x}} = C(\phi,\delta).
$$
The claim follows from the last two equations.
\end{proof}

The results of Lemma \ref{lem: rep cutoff} also hold when the cutoff is 
removed.
The proof shows (non-uniform) convergence of the solution of the F-KPP equation 
when one condition of Theorem \ref{bramson_fundamental_convergence} is not fulfilled.
With an appropriate continuity argument, a Lalley-Sellke type representation is recovered.
\begin{lem} \label{lifted}
Let $u(t,x), u_\de(t, x)$ be solutions of the F-KPP equation \eqref{kpp_one_minus} with initial condition $u(0,x)=1-e^{-\phi(-x)}$ 
and $u_\de(t, x)= 1-e^{-\phi(-x)}\1_{\{-x\leq \de\}}$, respectively. Set 
$C(\de, \phi) \defi \lim_{t\to\infty} \sqrt{\frac{2}{\pi}}\int_0^{\infty} u_\de(t,y+\sqrt{2}t)ye^{\sqrt{2}y}dy$. 
Then
$$
\lim_{t\to\infty} u(t,x+m(t))= 1-\E\left[ \exp -C(\phi)Ze^{-\sqrt{2}x}\right]\ ,
$$
with
$$
C(\phi)\defi \lim_{t\to\infty} \sqrt{\frac{2}{\pi}}\int_0^{\infty} u(t,y+\sqrt{2}t)ye^{\sqrt{2}y}dy = \lim_{\delta\to\infty} C(\phi,\delta) \ .
$$
\end{lem}
\begin{proof}
It is straightforward to check that
\beq \label{difference}
0\leq u_\delta(t, x)-u(t,x)\leq \PP(\max x_k(t)> \delta +x)\ , 
\eeq
from which it follows that
\beq \bea \label{bounds_u_l}
& \int_0^\infty u_\delta(t, x+\sqrt{2}t) x e^{\sqrt{2}x} \dd x - \int_0^\infty \PP\left[ \max x_k(t) -\sqrt{2} t> \delta +x \right] x e^{\sqrt{2}x} \dd x\\ 
& \hspace{2cm} \leq \int_0^\infty u(t, x+\sqrt{2}t) x e^{\sqrt{2}x} \dd x \\
& \hspace{4cm} \leq \int_0^\infty u_\delta(t, x+\sqrt{2}t) x e^{\sqrt{2}x} \dd x.
\eea \eeq
Define
\[
F(t, \delta) \defi \int_0^\infty u_\delta(t, x+\sqrt{2}t) x e^{\sqrt{2}x} \dd x,
\]
\[
U(t) \defi \int_0^\infty u(t, x+\sqrt{2}t) x e^{\sqrt{2}x} \dd x,
\]
and 
\[
M(t, \delta) \defi \int_0^\infty \PP\left[ \max x_k(t) -\sqrt{2} t> \delta +x \right] x e^{\sqrt{2}x} \dd x.
\]
The inequalities in \eqref{bounds_u_l} then read  
\beq \label{bounds_short}
F(t,\delta) -M(t, \delta) \leq F(t) \leq F(t, \delta).
\eeq
We claim that 
\beq \label{claim_one}
\lim_{\delta \to \infty} \lim_{t\to \infty} M(t, \delta) = 0.
\eeq
We postpone the proof of this, and remark that Proposition \ref{cr_our_setting}  implies that to given 
$\de$, $\lim_{t \to \infty} F(t, \delta) \defi F(\delta)$ exists and is strictly positive. 
We thus deduce from \eqref{bounds_short} that
\beq \label{bounds_inf_sup}
\liminf_{\delta \to \infty} F(\de) \leq \liminf_{t\to \infty} U(t) \leq \limsup_{t\to \infty} U(t) \leq \limsup_{\de\to \infty} F(\de).  
\eeq
We claim that $\lim_{\de\to \infty} F(\de)$ exists, is strictly positive and finite. To see this, 
we first observe that the function $\delta \to F(\de)$ is by construction decreasing, and positive, therefore the limit $\lim_{\de\to \infty} F(\de)$ exists. 
Strict positivity is proved in a somewhat indirect fashion: we proceed by contradiction, and rely on the convergence of the process of cluster extrema. Assume that
\beq \label{contradiction}
\lim_{\de \to \infty} F(\de) = 0,
\eeq 
and thus that
\beq \label{real_contradiction}
\lim_{t\to \infty} U(t) = 0.
\eeq
Using the form of the Laplace functional of a Poisson process, we have that
\beq \bea 
& E\left[\exp\left( - \int \phi(x) \Pi_t(\dd x) \right)\right] = \\
& = E \exp\left[-Z \sqrt{\frac{2}{\pi}}\int_{-\infty}^0 \left(1-\E\left[\exp-\sum_{k\leq n(t)}\phi(x+x_k(t)-\sqrt{2}t)\right]\right)\left\{-x e^{-\sqrt{2x}}\right\}dx\right] \\
& \stackrel{x\mapsto -x}{=} E \exp\left[-Z \sqrt{\frac{2}{\pi}} \int_{0}^\infty \left(1-\E\left[\exp-\sum_{k\leq n(t)}\phi(-x+x_k(t)-\sqrt{2}t)\right]\right) x e^{\sqrt{2x}} dx\right] \\
& = E\left[ \exp\big(-Z \sqrt{\frac{2}{\pi}} U(t) \big) \right].
\eea \eeq
Therefore, \eqref{real_contradiction} would imply that 
\beq \bea \label{laplace_cluster}
& \lim_{t \to \infty} E\left[\exp\left( - \int \phi(x) \Pi_t(\dd x) \right)\right] = E\left[\exp\left(-Z \sqrt{\frac{2}{\pi}} \lim_{t \to \infty} U(t) \right) \right] = 1\ .
\eea \eeq
This cannot hold. In fact, for $\Pi_t^{\text{ext}}$ the process of the cluster-extrema defined earlier,
one has the obvious bound 
\beq \label{dominated_pp}
E\left[\exp\left( - \int \phi(x) \Pi_t(\dd x) \right)\right] \leq E\left[\exp\left( - \int \phi(x) \Pi^{\text{ext}}_t(\dd x) \right)\right] \ .
\eeq
Since the process of cluster-extrema converges, by Proposition \ref{poissonianity_cluster_e}, to a PPP$\left( C Z \ee^{-\sqrt{2}x} \dd x\right)$, 
\beq
\lim_{t\to \infty} E\left[\exp\left( - \int \phi(x) \Pi_t(\dd x) \right)\right] \leq E\left[\exp\left( - C Z \int \left\{1-\ee^{-\phi(x)}\right\}  \ee^{-\sqrt{2}x} \dd x\right)\right] < 1.
\eeq
This contradicts \eqref{laplace_cluster} and therefore also \eqref{real_contradiction}. 

It remains to prove \eqref{claim_one}. 
For this, we shall use an upper bound for the right tail of the law of the maximum of BBM established in \cite{abk_poissonian}. 
It is a consequence of tight bounds established by Bramson \cite[Prop. 8.2]{bramson_monograph}.

\begin{lem}\cite[Cor. 10]{abk_poissonian} \label{tight_bounds}
For $X>1$, and $t\geq t_o$ (for $t_o$ a numerical constant),
\beq \bea \label{up_tight_zero}
\PP\left[ \max_{k\leq n(t)} x_k(t) -m(t) \geq X \right] & \leq \rho \cdot X \cdot \exp\left(-\sqrt{2} X - \frac{X^2}{2t}+\frac{3 }{2\sqrt{2}}X \frac{\log t}{t}\right).
\eea \eeq
for some constant $\rho>0$.
\end{lem}
Now since $\sqrt{2}t = m(t) + \frac{3}{2\sqrt{2}}\log t$,
\beq \bea \label{max_nocontr}
& \int_0^\infty \PP\left[ \max x_k(t) -\sqrt{2} t> \delta +x \right] x e^{\sqrt{2}x} \dd x \\
& \qquad = \int_0^\infty \PP\left[ \max x_k(t) - m(t) > \delta +x +\frac{3}{2\sqrt{2}} \log t \right] x e^{\sqrt{2}x} \dd x \\
& \qquad = \frac{e^{-\sqrt{2} \delta}}{t^{3/2}} \int_{\delta +\frac{3}{2\sqrt{2}} \log t}^\infty \PP\left[ \max x_k(t) - m(t) > y \right] y e^{\sqrt{2}y} \dd y  \\
& \qquad \qquad - \left(\de +\frac{3}{2\sqrt{2}} \log t \right)\frac{e^{-\sqrt{2} \delta}}{t^{3/2}} \int_{\delta +\frac{3}{2\sqrt{2}} \log t}^\infty \PP\left[ \max x_k(t) - m(t) > y \right] e^{\sqrt{2}y} \dd y,
\eea \eeq
the last line by change of variable. We address the large time limit of the first term on the right-hand side above.
The second term is handled similarly.
 By Lemma \ref{tight_bounds}, the first term is bounded, up to constant, by 
\beq \bea
& \frac{e^{-\sqrt{2} \delta}}{t^{3/2}} \int_{\delta +\frac{3}{2\sqrt{2}} \log t}^\infty y^2 \exp\left(-\frac{y^2}{2t} +\frac{3}{2\sqrt{2}} y\frac{\log t}{t}\right) \dd y \\
& \qquad = \frac{e^{-\sqrt{2} \delta}}{t^{3/2}} \exp\left[{\frac{9}{16} \frac{(\log t)^2}{t}}\right] \int_{\delta +\frac{3}{2\sqrt{2}} \log t}^\infty y^2 e^{-\frac{(y-3\log t/2\sqrt{2} )^2}{2t}} \dd y \\
& \qquad \leq 2\cdot  \frac{e^{-\sqrt{2} \delta}}{t^{3/2}} \int_{\delta }^\infty \left(z+\frac{3}{2\sqrt{2}} \log t\right)^2 e^{-z^2/2t} \dd z,
\eea \eeq
where in the last line we used that $\exp\left[{\frac{9}{16} \frac{(\log t)^2}{t}}\right] = 1+ o(1) \leq 2,$ as $t\to\infty$. 
By developing the square $\left(z+\frac{3}{2\sqrt{2}} \log t\right)^2$ one easily sees that the only contribution which is not vanishing in the limit comes from the $z^2$-term, for which we have
\beq 
2\cdot  \frac{e^{-\sqrt{2} \delta}}{t^{3/2}} \cdot \int_{\delta }^\infty z^2 e^{-z^2/2t} \dd z \leq \rho \cdot  e^{-\sqrt{2} \delta} \to 0, 
\eeq
as $\delta \to \infty$. This implies \eqref{claim_one} and concludes the proof of Lemma \ref{lifted}.
\end{proof}

Combining the assertions of Lemma  \ref{lem: rep cutoff} 
and Lemma \ref{lifted} 
yields the assertion of Proposition  \ref{eqn: laplace rep}.
\end{proof}

Proposition \ref{lifted} yields a short proof of Theorem \ref{main_theorem}.
\begin{proof}[Proof of Theorem \ref{main_theorem}] 
The Laplace functional of $\Pi_t$  using the form of the Laplace functional of a Poisson process reads
\beq \bea \label{laplace_cluster_two}
& E\left[\exp - \int \phi(x) \Pi_t(\dd x)\right]  \\
& = E\left[\exp- \int_{-\infty}^0 \left(1-\E\left[\exp-\sum_{k\leq n(t)}\phi(x+x_k(t)-\sqrt{2}t)\right]\right)\sqrt{\frac{2}{\pi}}\left\{-x e^{-\sqrt{2x}}\right\}dx\right] \\
& = E \exp\left[- \sqrt{\frac{2}{\pi}}\int_{0}^\infty u(t, x+\sqrt{2}t +\frac{1}{\sqrt{2}}\log Z) x e^{\sqrt{2x}} dx\right],
\eea \eeq
with 
$$
u(t, x)=1-\E\left[\exp-\sum_{k=1}^{n(t)}\phi(-x+x_k(t))\right].
$$
By  \eqref{eqn: shift},
\[
\lim_{t\to \infty}\sqrt{\frac{2}{\pi}} \int_{0}^\infty u(t, x+\sqrt{2}t+\frac{1}{\sqrt{2}}\log Z) x e^{\sqrt{2x}} dx= Z \sqrt{\frac{2}{\pi}} \lim_{t\to \infty} \int_{0}^\infty u(t, x+\sqrt{2}t) x e^{\sqrt{2x}} dx\ ,
\]
and the limit exists and is strictly positive by Proposition \ref{lem: rep cutoff}. 
This implies that the Laplace functionals of $\lim_{t\to\infty}\Pi_t$ and of the extremal process of BBM are equal.
The proof of Theorem \ref{main_theorem} is concluded.  
\end{proof}

\subsection{Properties of the clusters}
In this section we prove Proposition \ref{squareroot_prop} and Proposition \ref{iid}.

\begin{proof}[Proof of Proposition \ref{squareroot_prop}]
Throughout the proof, the probabilities are considered conditional on $Z$.
We show that for $\vare>0$ there exists $C_1, C_2$ such that 
\beq \label{starting_square}
\sup_{t\geq t_0} P\left[\exists_{i,k}: \eta_i + \frac{1}{\sqrt{2}}\log Z + x_k^{(i)}(t)-\sqrt{2} t  \geq Y,\, \text{but}\: \eta_i \notin [-C_1 \sqrt{t}, -C_2 \sqrt{t}] \right] \leq \vare.
\eeq
The proof is split in two parts. We claim that, for $t$ large enough, there exists $C_1>0$ small enough such that 
\beq \label{square_one}
P\left[\exists_{i,k}: \eta_i + \frac{1}{\sqrt{2}}\log Z+ x_k^{(i)}(t)-\sqrt{2} t  \geq Y,\, \text{but}\: \eta_i \geq -C_1 \sqrt{t}  \right] \leq \vare/2, 
\eeq 
and $C_2>0$ large enough such that 
\beq \label{square_two}
P\left[\exists_{i,k}: \eta_i + \frac{1}{\sqrt{2}}\log Z+ x_k^{(i)}(t)-\sqrt{2} t  \geq Y,\, \text{but}\: \eta_i \leq -C_2 \sqrt{t}  \right] \leq \vare/2. 
\eeq 
By Markov inequality the left-hand side of \eqref{square_one} is less than
\beq \bea \label{square_three}
& \int_{-C_1 \sqrt{t}}^0 P\left[ \max_{k} x_k(t) \geq \sqrt{2} t+ Y -x-  \frac{1}{\sqrt{2}}\log Z\right] (-x \ee^ {-\sqrt{2}x}) \dd x \\
& \quad = \int_{0}^{C_1 \sqrt{t}} P\left[ \max_{k} x_k(t) \geq \sqrt{2} t+ Y +x- \frac{1}{\sqrt{2}}\log Z\right] x \ee^ {\sqrt{2}x} \dd x
\eea \eeq
This is the integral appearing in \eqref{max_nocontr}, truncated at $C_1 \sqrt{t}$, and with $\de$ replaced by $Y$. Hence, as $t\to \infty$,  
\beq \bea \label{square_four}
\eqref{square_three} & = \frac{e^{-\sqrt{2} Y}Z}{t^{3/2}} \int_{Y +\frac{3}{2\sqrt{2}} \log t}^{C_1 \sqrt{t}+\frac{3}{2\sqrt{2}}\log t} x^2 \exp\left(-\frac{x^2}{2t} +\frac{3}{2\sqrt{2}} x\frac{\log t}{t}\right) \dd x+ o(1).
\eea \eeq
By change of variable $x \to \frac{x}{\sqrt{t}}$ we see that, as $t\to \infty$ and up to irrelevant numerical constants, 
\beq \label{square_five}
\eqref{square_four} \leq Ze^{-\sqrt{2} Y} (1+o(1)) \int_0^{C_1} x^2  \ee^{-x^2/2} \dd x.
\eeq
But for smaller and smaller $C_1$ the integral is obviously vanishing: it thus suffices to choose $C_1$ small enough to have that \eqref{square_five} is less than $\vare/2$, settling \eqref{square_one}.

The proof of \eqref{square_two} is analogous, and we omit the details. The end result is that, for large enough $t$ and up to irrelevant \
numerical constant, 
\beq \label{square_final}
\eqref{square_two} \leq Ze^{-\sqrt{2} Y} (1+o(1)) \int_{C_2}^\infty x^2  \ee^{-x^2/2} \dd x.
\eeq
It suffices to choose $C_2$ large enough to obtain \eqref{square_two}.
\end{proof}
 
For the proof of Proposition \ref{iid}, the following Lemma is needed. 

\begin{lem} \label{not_a}
Let $u(t,x)$ be a solution to the F-KPP equation \eqref{kpp_one_minus} with initial data satisfying
$$
\int_0^\infty y \ee^{\sqrt{2}y}u(0,y)dy<\infty\ ,
$$
and such that $u(t,\cdot+m(t))$ converges. Let  $\psi$ be the associated 
approximation as in Proposition \ref{bramson_fundamental_interpolation}. Then, for  $x = a \sqrt{t}$ and $Y\in \R$,  
\beq \label{limit_k}
\lim_{t\to \infty} \frac{\ee^{\sqrt{2}x} t^{3/2}}{x} \psi(r, t, x+Y +\sqrt{2}t) = K \ee^{-\sqrt{2} Y} \ee^{-a^2/2}
\eeq
where
\[
K = \sqrt{\frac{2}{\pi}}\int_0^\infty u(r,y'+\sqrt{2}r) y' e^{y'\sqrt{2}} \dd y' \ .
\]
Moreover, the convergence is uniform for $a$ in a compact set.
\end{lem}
\begin{proof}
The proof is a simple computation: 
\beq \bea \label{bounded_conv_ind}
& \lim_{t\to \infty}  \frac{\ee^{\sqrt{2}x}}{x} t^{3/2}   \psi(r, t, Y+x +\sqrt{2}t) \\
& \stackrel{\eqref{psi_no_m}}{=} \ee^{-\sqrt{2}Y} \lim_{t\to \infty} \frac{t^{3/2}}{x \sqrt{2\pi(t-r)}}\int_0^\infty u(r,y'+\sqrt{2}r)e^{y'\sqrt{2}}e^{-\frac{(y'-x-Y)^2}{2(t-r)}}(1-e^{-2y'\frac{(x+Y+\frac{3}{2\sqrt{2}}\log t)}{t-r}})dy'\\
& = \ee^{-\sqrt{2}Y} \int_0^\infty u(r,y'+\sqrt{2}r)e^{y'\sqrt{2}}\lim_{t\to \infty}  \frac{t^{3/2}}{x \sqrt{2\pi(t-r)}}\left[e^{-\frac{(y'-x-Y)^2}{2(t-r)}}(1-e^{-2y'\frac{(x+Y+\frac{3}{2\sqrt{2}}\log t)}{t-r}})\right]dy',
\eea \eeq
the last step by dominated convergence (cfr. \eqref{dominated_one}-\eqref{dominated_five}). 
Using that $x = a\sqrt{t}$,
\beq 
\lim_{t\to \infty}  \frac{t^{3/2}}{x \sqrt{2\pi(t-r)}}\left[e^{-\frac{(y'-x-Y)^2}{2(t-r)}}(1-e^{-2y'\frac{(x+Y+\frac{3}{2\sqrt{2}}\log t)}{t-r}})\right] = \sqrt{\frac{2}{\pi}} y' \ee^{-a^2/2},  
\eeq
hence
\beq \label{bounded_conv_upshot}
\lim_{t\to \infty} \frac{\ee^{\sqrt{2}x}}{x} t^{3/2}   \psi(r, t, Y+x +\sqrt{2}t) = K \ee^{-\sqrt{2}Y}\ee^{-a^2/2}\ .
\eeq 
\end{proof}

\begin{proof}[Proof of Proposition \ref{iid}]
Let $a\in[-C_1,-C_2]$ and $b\in\R$. Set $x=a\sqrt{t}+b$.
Define for convenience
$$
\overline{\mathcal{E}}_t\defi\sum_{i}\delta_{x_i(t)-\sqrt{2}t}\ ,
$$
and $\max \overline{\mathcal{E}}_t\defi \max_i x_i(t)-\sqrt{2}t$.
We first claim that $  x + \max \overline{\mathcal{E}}_t$ conditionally on $\{ x + \max \overline{\mathcal{E}}_t >0\}$ weakly
converges to an exponential random variable, 
\beq \label{exponential_rv}
\lim_{t\to\infty}\PP\left[ x + \max \overline{\mathcal{E}}_t >X  ~ \Big| x + \max \overline{\mathcal{E}}_t >0\right] = \ee^{-\sqrt{2} X},
\eeq
for $X>0$ (and $0$ otherwise). Remark, in particular, that the limit does not depend on $x$. 

To see \eqref{exponential_rv}, we write the conditional probability as
\beq
\label{max_cond_survival_two}
\frac{\PP\left[ x + \max \overline{\mathcal{E}}_t >X \right]}{ \PP\left[ x + \max \overline{\mathcal{E}}_t >0\right]} \,.
\eeq
For $t$ large enough (and hence $-x$ large enough in the positive) we may apply the uniform bounds from Proposition \ref{bramson_fundamental_interpolation} in the form 
\beq \label{upper_psi_ind}
\PP\left[ \max_{k\leq n(t)} x_k(t) \geq X-x +\sqrt{2}t \right] \leq \gamma(r) \psi(r, t, X-x +\sqrt{2}t)
\eeq
and 
\beq \label{lower_psi_ind}
\PP\left[ \max_{k\leq n(t)} x_k(t) \geq -x +\sqrt{2}t \right] \geq \gamma^{-1}(r) \psi(r, t, -x +\sqrt{2}t)
\eeq
where $\psi$ is as in \eqref{psi_no_m} and the $u$ entering into its definition is solution to F-KPP with Heaviside initial conditions, and $r$ is large enough. Therefore, 
\beq \bea \label{max_cond_survival_three}
\gamma^{-2}(r) \frac{\psi(r, t, X-x +\sqrt{2}t)}{\psi(r, t, -x +\sqrt{2}t)} \leq \eqref{max_cond_survival_two} \leq \gamma^2(r) \frac{\psi(r, t, X-x +\sqrt{2}t)}{\psi(r, t, -x +\sqrt{2}t)}.
\eea \eeq
By Lemma \ref{not_a}, 
\beq \bea
\lim_{t\to \infty} \frac{\psi(r, t, X-x +\sqrt{2}t)}{\psi(r, t, -x+\sqrt{2}t)} = \ee^{-\sqrt{2} X}.
\eea\eeq
Taking the limit $t\to \infty$ first, and then $r\to \infty$ (and using that $\gamma(r) \downarrow 1$) we thus see that \eqref{max_cond_survival_three} implies 
\eqref{exponential_rv}. 

Second, we show that for any function $\phi$ that is continuous with compact support, the limit of 
$$
\E\left[\exp -\int \phi(x+z) \overline{\mathcal{E}}_t(dz) \Big| x + \max \overline{\mathcal{E}}_t>0\right]
$$
exists and is independent of $x$. It follows from the first part of the proof that the conditional process has a maximum almost surely. 
It is thus sufficient to consider the truncated Laplace functional, that is for 
$\delta>0$, 
\beq
\label{eqn: laplace cond trunc}
\E\left[\exp\left( -\int \phi(x+z) \overline{\mathcal{E}}_t(dz)\right)\1_{\{x+ \max \overline{\mathcal{E}}_t \leq \delta\}} \Big| x+ \max \overline{\mathcal{E}}_t>0\right]
\eeq
The above conditional expectation can be written as 
\beq \bea \label{laplace_ind_a_three}
& = \frac{\E\left[ \prod_{k=1}^{n(t)}  \exp\Big(-\phi(x+x_k(t)-\sqrt{2}t)\Big) \1_{\{x+ x_k(t)-\sqrt{2}t \leq \delta \}} \right]}{ \PP\left[x+  \max \overline{\mathcal{E}}_t >0 \right]} \\
& \qquad \qquad- \frac{\E\left[ \prod_{k=1}^{n(t)}  \exp\Big(-\phi(x+x_k(t)-\sqrt{2}t)\Big) \1_{\{x+x_k(t)-\sqrt{2}t \leq 0 \}} \right]}{\PP\left[ x+  \max \overline{\mathcal{E}}_t >0 \right]}\ .
\eea \eeq
Define
\[
\bea
 u_1(t, y ) &\defi 1-\E\left[ \prod_{k=1}^{n(t)} \ee^{-\phi(-y+x_k(t))} \1_{\{-y+ x_k(t)\leq 0\}} \right]\\
u_2(t, y) &\defi 1-\E\left[ \prod_{k=1}^{n(t)} \ee^{-\phi(-y+x_k(t))} \1_{\{-y+x_k(t) \leq \delta \}} \right]\\
u_3(t, y) &\defi  \PP\left[ -y+\max_k x_k(t) \leq 0 \right]\\
\eea
\]
so that 
\beq \bea \label{laplace_ind_a_four}
\eqref{laplace_ind_a_three} 
& = \frac{u_2(t, -x +\sqrt{2} t)}{u_3(t, -x+\sqrt{2} t)} - \frac{u_1(t, -x+\sqrt{2} t)}{u_3(t, -x+\sqrt{2} t)}.
\eea \eeq

Remark that the functions $u_1, u_2$ and $u_3$, all solve the F-KPP equation \eqref{kpp_one_minus} with initial conditions
\beq \bea
\label{eqn: u_i initial}
u_1(0, y) &= 1-\ee^{-\phi(-y)} \1_{\{-y \leq 0\}}, \\
u_2(0, y) &= 1-\ee^{-\phi(-y)} \1_{\{-y \leq \delta\}}, \\
u_3(0, y) &= 1- \1_{\{-y \leq 0\}}.
\eea \eeq
They also satisfy the assumptions of Theorem \ref{bramson_fundamental_convergence} and Proposition \ref{bramson_fundamental_interpolation}.

Let $\psi_i$ be as in \eqref{psi_no_m} with $u$ replaced by the appropriate 
$u_i$, $i=1,2,3$.
By Proposition \ref{bramson_fundamental_interpolation}, 
\beq \bea \label{upshot_iid}
& \lim_{t\to \infty} \frac{u_2(t, -x +\sqrt{2} t)}{u_3(t, -x+\sqrt{2} t)} -
 \frac{u_1(t, -x+\sqrt{2} t)}{u_3(t, -x+\sqrt{2} t)}  \\
& \quad =  \lim_{r\to \infty} \lim_{t\to \infty} 
\left\{ \frac{\psi_2(r,t, -x +\sqrt{2} t)}{\psi_3(r,t, -x+\sqrt{2} t)}
 \right\} - \lim_{r\to \infty} \lim_{t\to \infty}
 \left\{ \frac{\psi_1(r,t, -x+\sqrt{2} t)}
{\psi_3(r, t, -x+\sqrt{2} t)} \right\}.
\eea \eeq
By Lemma \ref{not_a}, the above limits exist and do not depend on $x$. 
This shows the existence of \eqref{eqn: laplace cond trunc}.
It remains to prove that this limit is non-zero for a non-trivial $\phi$, thereby showing the existence and local finiteness of the conditional point process.
To see this, note that, by Proposition \ref{cr_our_setting}, the limit of \eqref{eqn: laplace cond trunc} equals
$$
\frac{C_2-C_1}{C_3}
$$
where $C_i=\lim_{t\to\infty}\int_{-\infty}^0 u_i(t,y'+\sqrt{2}t)(-y'e^{-\sqrt{2}y'}dy')$ for $u_i$ as above.
Note that $0<C_3<\infty$, since by the representation of Theorem \ref{teor_existence}
$$
\lim_{t\to\infty}\PP(\max x_i(t)-m(t) \leq z)= \E \exp -C_3 Z e^{-\sqrt{2}z},
$$
and this probability is non-trivial. Now suppose $C_1=C_2$.
Then by Theorem \ref{teor_existence} again, this would entail
$$
\lim_{t\to\infty}\E\left[ \left(\exp-\int \phi(x) \mathcal{E}_t(dx) \right)\1_{\{\max \mathcal{E}_t<\delta\}}\right]=\lim_{t\to\infty}\E\left[ \left(\exp-\int \phi(x) \mathcal{E}_t(dx) \right)\1_{\{\max \mathcal{E}_t<0\}}\right]\ ,
$$
where $\mathcal{E}_t=\sum_i \delta_{x_i(t)-m(t)}$ and $\max \mathcal{E}_t=\max_i x_i(t)-m(t)$. Thus,
$$
\lim_{t\to\infty}\E\left[ \left(\exp-\int \phi(x) \mathcal{E}_t(dx) \right)\1_{\{0<\max \mathcal{E}_t<\delta\}}\right]= 0\ .
$$
But this is impossible since the maximum has positive probability of occurrence in $[0,\delta]$ for any $\delta$ and
the process $\lim_{t\to\infty}\mathcal{E}_t$ is locally finite. This concludes the proof of the Proposition.
\end{proof}

Define the gap process at time $t$
\begin{equation}
\mathcal D_t\defi \sum_i \delta_{x_i(t)-\max_j x_j(t)}\ .
\label{eqn: Delta}
\end{equation}
Let us write $\overline{\mathcal{E}}$ for the point process obtained in Proposition \ref{iid} from the limit of the conditional
law of $\overline{\mathcal{E}}_t$ given $\max \overline{\mathcal{E}}_t>0$. 
We denote by $\max \overline{\mathcal{E}}$ the maximum of $\overline{\mathcal{E}}$, and
by $\mathcal D$ the process of the gaps of $\overline{\mathcal{E}}$, that is the process $\overline{\mathcal{E}}$ shifted back by  $\max \overline{\mathcal{E}}$.
The following corollary is the fundamental result showing that $\mathcal D$ is the limit of the conditioned process $\mathcal D_t$, and, perhaps surprisingly, 
the process of the gaps in the limit is independent of the location of the maximum.
\begin{cor}
\label{cor: gaps}
Let $x=a\sqrt{t}$, $a<0$.
In the limit $t\to\infty$, the random variables $\mathcal D_t$ and $x+\max \overline{\mathcal{E}}$ 
are conditionally independent on the event $x+\max \overline{\mathcal{E}}>b$ for any $b\in\R$.
More precisely, for any bounded continuous function $f,h$ and $\phi\in \mathcal{C}_c(\R)$,
$$
\bea
&\lim_{t\to\infty}\E \left[ f\left(\int \phi(z) \mathcal D_t(dz)\right) h(x+\max \overline{\mathcal{E}}_t) \Big| x+\max \overline{\mathcal{E}}_t >b\right]\\
&\qquad\qquad \qquad=\E\left[ f\left(\int \phi(z) \mathcal D(dz)\right)\right] \int_{b}^\infty h(y) \frac{\sqrt{2}e^{-\sqrt{2}y}dy}{e^{-\sqrt{2}b}}\ .
\eea
$$
Moreover, convergence is uniform in $x=a\sqrt{t}$ for $a$ in a compact set. 
\end{cor}
\begin{proof}
By standard approximation, it suffices to establish the result for $h(y)=\1_{\{y>b'\}}$ for $b'>b$.
By the property of conditioning and since $b'>b$
$$
\bea
&\E \left[ f\left(\int \phi(z) \mathcal D_t(dz)\right) 
\1_{\{x+\max\mathcal{E}_t>b'\}} \Big| x+\max \overline{\mathcal{E}}_t >b\right]\\
&\qquad=\E \left[ f\left(\int \phi(z) \mathcal D_t(dz)\right) \Big| x-b'+\max \overline{\mathcal{E}}_t>0\right]\frac{\PP\left[ x-b+\max \overline{\mathcal{E}}_t >b'-b\right]}{\PP\left[ x-b+\max \overline{\mathcal{E}}_t >0\right]}\ .
\eea
$$
The conclusion will follow from Proposition \ref{iid} by taking the limit $t\to\infty$, once it is shown that convergence of $(\overline{\mathcal{E}}_t, y+ \max \overline{\mathcal{E}}_t)$ under the conditional law implies convergence of the gap process $\mathcal{D}_t$.
This is a general continuity result which is done in the next lemma.
\end{proof}

\begin{lem}
\label{lem: continuity}
Let $(\mu_t,X_t)$ be a sequence of random variables on $\mathcal{M}\times \R$ that converges to $(\mu,X)$ in the sense that
for any bounded continuous function $f,h$ on $\R$ and any $\phi\in\mathcal{C}_c(\R)$
$$
\E\left[ f\left(\int \phi d\mu_t\right)h(X_t)\right]\to \E\left[ f\left(\int \phi d\mu\right)h(X)\right]\ .
$$
Then for any $\phi\in\mathcal{C}_c(\R)$ and $g:\R\to\R$, bounded continuous,
$$
\E\left[g\left(\int \phi(y+X_t)\mu_t(dy)\right)\right]\to \E\left[g\left(\int \phi(y+X)\mu(dy)\right)\right] \ .
$$
\end{lem}
\begin{proof}
Let $f:\R\to\R$ be a bounded continuous function. 
Introduce the notation
$$
T_x\mu(\phi)\defi \int \phi(y+x) \mu(dx)\ .
$$
We need to show that for $t$ large enough
\beq
\label{eqn: diff}
\Big|\E \left[ f\left(T_{X_t}\mu_t(\phi)\right)\right]-\E \left[ f\left(T_{X}\mu(\phi)\right)\right]\Big|
\eeq
is smaller than $\vare$. By standard approximations, it is enough to suppose $f$ is Lipschitz, whose constant we assume to be $1$ for simplicity.

Since the random variables $(X_t)$ are tight by assumption, there exist $t(\vare)$ large enough and $K_\vare$, an interval of $\R$, such that
$$
\eqref{eqn: diff}\leq \Big|\E \left[ f\left(T_{X_t}\mu_t(\phi)\right);X_t\in K_\vare\right]-\E \left[ f\left(T_{X}\mu(\phi)\right);X\in K_\vare\right]\Big| +\vare\ .
$$
Now divide $K_\vare$ into $N$ intervals $I_j$ of equal length. Write $\bar{x}_j$ for the midpoint of $I_j$.
For each of these intervals, one has
\beq
\label{eqn: approx X_t}
\E \left[ f\left(T_{X_t}\mu_t(\phi)\right); X_t\in I_j\right]= \E \left[ f\left(T_{\bar{x}_j}\mu_t(\phi)\right); X_t\in I_j\right]
+ \mathcal{R}(t,j)
\eeq
for
$$
\mathcal{R}(t,j)\leq \E \left[ |f\left(T_{X_t}\mu_t(\phi)\right)- f\left(T_{\bar{x}_j}\mu_t(\phi)\right)| ; X_t\in I_j\right]
$$
Since $f$ is Lipschitz, the right-hand side is smaller than
\beq
\label{eqn: lip}
 \E\left[|T_{X_t}\mu_t(\phi)- T_{\bar{x}_j}\mu_t(\phi)| ; X_t\in I_j\right] \ .
\eeq
Moreover
$$
|T_{X_t}\mu_t(\phi)- T_{\bar{x}_j}\mu_t(\phi)|=\int|\phi(y-X_t)-\phi(y-\bar{x}_j)| \mu_t(dy)\ .
$$
Note that there exists a compact C, independently of $t$ and $j$ so that $|\phi(y-X_t)-\phi(y-\bar{x}_j)|=0$ for $y\notin C$.
(It suffices to take $C$ so that it contains all the translates $\text{supp} \phi +k$, $k\in K_\vare$). 
By taking $N$ large enough, $|y-X_t- (y-\bar{x}_j)|=|\bar{x}_j-X_t|<\delta_\phi$ for the appropriate $\delta_\phi$
making $|\phi(y-X_t)-\phi(y-\bar{x}_j)|<\vare$, uniformly on $y\in C$. 
Hence, \eqref{eqn: lip} is smaller than
$$
\vare \E[\mu_t(C); X_t\in I_j]
$$
The summation over $j$ is thus smaller than $\vare \E[\mu_t(C)]$. 
By the convergence of $(\mu_t)$, this can be made smaller for $t$ large enough.

The same approximation scheme for $(\mu,X)$ yields
\beq
\label{eqn: approx X}
\E \left[ f\left(T_{X}\mu(\phi)\right); X_t\in I_j\right]= \E \left[ f\left(T_{\bar{x}_j}\mu(\phi)\right); X\in I_j\right]
+ \mathcal{R}(j)
\eeq
where $\sum_j \mathcal{R}(j)\leq \vare \E[\mu(C)]$.
Therefore \eqref{eqn: diff} will hold provided that the difference of the first terms of the right-hand side of \eqref{eqn: approx X_t}
and of  \eqref{eqn: approx X} is small for $t$ large enough and $N$ fixed. But this is guaranteed by the hypotheses 
on the convergence of $(\mu_t,X_t)$.
\end{proof}

\subsection{Characterisation of the extremal process}

\begin{proof}[Proof of Theorem \ref{teor main}]
It suffices to show that for $\phi: \R \to \R_+$ continuous with compact support the Laplace functional of the extremal process of branching 
Brownian motion satisfies  
\beq \label{equality_laplace_t}
\lim_{t\to \infty} \Psi_t(\phi) = \E \exp\left(-CZ\int_\R \E[1- e^{-\int \phi(y+z) \mathcal D(dz)}] \sqrt{2} e^{-\sqrt{2}y} dy\right)
\eeq 
for the point process $\mathcal D$ of Corollary \ref{cor: gaps}. 

Now, by Theorem \ref{main_theorem},
\beq \bea \label{start}
& \lim_{t\to \infty} \Psi_t(\phi) = \lim_{t\to\infty} \E\left[\exp -\sum_{i,k}\phi(\eta_i +\frac{1}{\sqrt{2}}\log Z+x_k^{(i)}(t)-\sqrt{2}t ) \right]. 
\eea \eeq 
Using the form for the Laplace transform of a Poisson process we have for the r.h.s. above
\beq \bea \label{main_one} 
&\lim_{t\to\infty} \E\left[\exp -\sum_{i,k}\phi(\eta_i +\frac{1}{\sqrt{2}}\log Z+x_k^{(i)}(t)-\sqrt{2}t ) \right] \\
& \qquad =\E\exp\left(-Z\lim_{t\to\infty}\int_{-\infty}^0 \E\left[1- \exp-\int \phi(x+y)\overline{\mathcal{E}}_t(dx)\right]\sqrt{\frac{2}{\pi}}(-ye^{-\sqrt{2}y})dy\right). \\
\eea \eeq
Let $\mathcal D_t$ as in \eqref{eqn: Delta}. The integral of the right-hand side above can be written as
$$
\lim_{t\to\infty}\int_{-\infty}^0  \E\left[f\left(\int \left\{T_{y+\max \overline{\mathcal{E}}_t}\phi(z)\right\} \mathcal D_t(dz) \right)\right]\sqrt{\frac{2}{\pi}}(-ye^{-\sqrt{2}y})dy
$$
for the bounded (on $[0,\infty)$) continuous function $f(x)=1-e^{-x}$, and where $T_x\phi(y)=\phi(y+x)$.
By Proposition \ref{squareroot_prop}, there exist $C_1$ and $C_2$ such that
\beq
\label{eqn: int approx}
\bea
& \int_{-\infty}^0 \E\left[f\left(\int \left\{T_{y+\max \overline{\mathcal{E}}_t}\phi(z)\right\} \mathcal D_t(dz) \right)\right]\sqrt{\frac{2}{\pi}}(-ye^{-\sqrt{2}y})dy \\
& \qquad = \Omega_t(C_1, C_2) + \int_{-C_2 \sqrt{t}}^{-C_1 \sqrt{t}} \E\left[f\left(\int \left\{T_{y+\max \overline{\mathcal{E}}_t}\phi(z)\right\} \mathcal D_t(dz) \right)\right]\sqrt{\frac{2}{\pi}}(-ye^{-\sqrt{2}y})dy, \\
\eea
\eeq
where the error term satisfies $\lim_{C_1\downarrow 0, C_2 \uparrow \infty} \sup_{t\geq t_0}\, \Omega_t(C_1, C_2) = 0.$
Introducing a conditioning on the event $y+\max \overline{\mathcal{E}}_t> m_\phi$,
the term in the integral becomes
\beq
\label{eqn: integrand}
\bea
& \E\left[f\left(\int \left\{T_{y+\max \overline{\mathcal{E}}_t}\phi(z)\right\} \mathcal D_t(dz) \right)\right]= \\
&\qquad \E\left[f\left(\int \left\{T_{y+\max \overline{\mathcal{E}}_t}\phi(z)\right\} \mathcal D_t(dz) \right) \Big| y+\max \overline{\mathcal{E}}_t> m_\phi\right] \PP\left[y+\max \overline{\mathcal{E}}_t> m_\phi\right]\ .
 \eea
\eeq
By Corollary \ref{cor: gaps}, the conditional law of the pair $\mathcal D_t, y+\max \overline{\mathcal{E}}_t$ {\it given} $\{y+\max \overline{\mathcal{E}}_t> m_\phi\}$ exists in the limit.
Moreover the convergence is uniform in $y \in [-C_1 \sqrt{t}, -C_2 \sqrt{t}]$.
By Lemma \ref{lem: continuity}, the convergence applies to the random variable $\int \left\{T_{y+\max \overline{\mathcal{E}}_t}\phi(z)\right\} \mathcal D_t(dz)$. 
Therefore
\beq
\label{eqn: convergence int}
\bea
&\lim_{t\to\infty}  \E\left[f\left(\int \left\{T_{y+\max \overline{\mathcal{E}}_t}\phi(z)\right\} \mathcal D_t(dz) \right) \Big| y+\max \overline{\mathcal{E}}_t> m_\phi\right]\\
&\qquad\qquad
=\int_{m_\phi}^\infty\E\left[f\left(\int (T_y\phi(z) \mathcal D(dz) \right)\right] \frac{\sqrt{2}e^{-\sqrt{2}y}dy}{e^{-\sqrt{2}m_\phi}}
\eea
\eeq
On the other hand,
\beq
\label{eqn: max}
\int_{-C_2 \sqrt{t}}^{-C_1 \sqrt{t}}  \PP\left[y+\max \overline{\mathcal{E}}_t> m_\phi\right] \sqrt{\frac{2}{\pi}}(-ye^{-\sqrt{2}y})dy= C e^{-\sqrt{2}m_\phi}+\Omega_t(C_1, C_2)
\eeq
by Lemma \ref{lem: max convergence} and by the same approximation as in \eqref{eqn: int approx}.

Combining \eqref{eqn: max}, \eqref{eqn: convergence int} and \eqref{eqn: integrand} gives that \eqref{main_one} converges  to 
\[
 \E \exp\left(-CZ\int_\R \E[1- e^{-\int \phi(y+z) \mathcal D(dz)}]\sqrt{2}e^{-\sqrt{2}y} dy\right) \, ,
\]
which is by \eqref{start} also the limiting Laplace transform of the extremal process of branching Brownian motion: this shows \eqref{equality_laplace_t}
and thus concludes the proof of Theorem \ref{teor main}. 

\end{proof}

\end{document}

%% file: picture_two.pstex_t
\begin{picture}(0,0)%
\includegraphics{picture_two.pstex}%
\end{picture}%
\setlength{\unitlength}{2072sp}%
\begingroup\makeatletter\ifx\SetFigFontNFSS\undefined%
\gdef\SetFigFontNFSS#1#2#3#4#5{%
  \reset@font\fontsize{#1}{#2pt}%
  \fontfamily{#3}\fontseries{#4}\fontshape{#5}%
  \selectfont}%
\fi\endgroup%
\begin{picture}(13196,7839)(349,-8608)
\put(10531,-2221){\makebox(0,0)[lb]{\smash{{\SetFigFontNFSS{12}{14.4}{\rmdefault}{\mddefault}{\updefault}{\color[rgb]{0,0,0}$m(t)$}%
}}}}
\put(9811,-7891){\makebox(0,0)[lb]{\smash{{\SetFigFontNFSS{12}{14.4}{\rmdefault}{\mddefault}{\updefault}{\color[rgb]{0,0,0}$t$}%
}}}}
\put(3511,-7801){\makebox(0,0)[lb]{\smash{{\SetFigFontNFSS{12}{14.4}{\rmdefault}{\mddefault}{\updefault}{\color[rgb]{0,0,0}$r_d$}%
}}}}
\put(10531,-2896){\makebox(0,0)[lb]{\smash{{\SetFigFontNFSS{12}{14.4}{\rmdefault}{\mddefault}{\updefault}{\color[rgb]{0,0,0}extremal particles}%
}}}}
\put(721,-2041){\rotatebox{90.0}{\makebox(0,0)[lb]{\smash{{\SetFigFontNFSS{12}{14.4}{\rmdefault}{\mddefault}{\updefault}{\color[rgb]{0,0,0}space}%
}}}}}
\put(11161,-7711){\makebox(0,0)[lb]{\smash{{\SetFigFontNFSS{12}{14.4}{\rmdefault}{\mddefault}{\updefault}{\color[rgb]{0,0,0}time}%
}}}}
\put(4636,-1546){\makebox(0,0)[lb]{\smash{{\SetFigFontNFSS{12}{14.4}{\rmdefault}{\mddefault}{\updefault}{\color[rgb]{0,0,0}$s\mapsto \frac{s}{t}m(t)$}%
}}}}
\put(1576,-3256){\makebox(0,0)[lb]{\smash{{\SetFigFontNFSS{12}{14.4}{\rmdefault}{\mddefault}{\updefault}{\color[rgb]{0,0,0}$E_{\alpha,t}(s)$: entropic envelope}%
}}}}
\put(8416,-7936){\makebox(0,0)[lb]{\smash{{\SetFigFontNFSS{12}{14.4}{\rmdefault}{\mddefault}{\updefault}{\color[rgb]{0,0,0}$t-r_g$}%
}}}}
\put(5131,-8476){\makebox(0,0)[lb]{\smash{{\SetFigFontNFSS{12}{14.4}{\rmdefault}{\mddefault}{\updefault}{\color[rgb]{0,0,0}no branching}%
}}}}
\end{picture}%